\documentclass[leqno,12pt]{article} %leqno is the option to put formula numbers on the left side
\usepackage{amsmath, amssymb}
\usepackage{amsthm} %theorem environment option
\theoremstyle{plain} %text of this environment is typesetted in italics
\newtheorem{theorem}{\indent\sc Theorem}[section]
\newtheorem{lemma}[theorem]{\indent\sc Lemma}
\newtheorem{corollary}[theorem]{\indent\sc Corollary}
\newtheorem{proposition}[theorem]{\indent\sc Proposition}

\theoremstyle{definition} %text of this environment is typesetted in roman letters
\newtheorem{definition}[theorem]{\indent\sc Definition}
\newtheorem{remark}[theorem]{\indent\sc Remark}

\title{Canonical connections attached to generalized quaternionic and para-quaternionic structures}
\author{Adara M. Blaga and Antonella Nannicini}
\date{}
\begin{document}

\maketitle

\markboth{{\small\it {\hspace{0.5cm} Canonical connections attached to generalized quaternionic and para-quaternionic structures}}}{\small\it{Canonical connections attached to generalized quaternionic and para-quaternionic structures\hspace{0.5cm}}}

%%%%%%%%%%%%%%% footnote %%%%%%%%%%%%%%%%


\begin{abstract}
We put into light some generalized almost quaternionic and almost para-quater\-nion\-ic structures and characterize their integrability with respect to a $\nabla$-bracket on the generalized tangent bundle $TM\oplus T^*M$ of a smooth manifold $M$, defined by an affine connection $\nabla$ on $M$. Also, we provide necessary and sufficient conditions for these structures to be $\hat \nabla$-parallel and $\hat \nabla^*$-parallel, where $\hat \nabla$ is an affine connection on $TM\oplus T^*M$ induced by $\nabla$, and $\hat\nabla^*$ is its generalized dual connection with respect to a bilinear form $\check h$ on $TM\oplus T^*M$ induced by a non-degenerate symmetric or skew-symmetric $(0,2)$-tensor field $h$ on $M$. As main results, we establish the existence of a canonical connection associated to a generalized quaternionic and to a generalized para-quaternionic structure, i.e., a torsion-free generalized affine connection that parallelizes these structures. We show that, in the quaternionic case, the canonical connection is the generalized Obata connection and that on a quasi-statistical manifold $(M,h,\nabla)$, an integrable $h$-symmetric and $\nabla$-parallel $(1,1)$-tensor field gives rise to a generalized para-quaternionic structure whose canonical connection is precisely $\hat \nabla^*$. Finally we prove that the generalized affine connection that parallelizes certain families of generalized almost complex and almost product structures is preserved.
\end{abstract}

\section{Introduction}

An important problem in the general theory of connections is to construct an affine connection with respect to which a certain tensor field is covariantly constant.
The basic example is the Levi-Civita connection for a pseudo-Riemannian metric tensor field. For an almost complex or an almost product structure,
such connections that parallelizes them have been determined by Obata \cite{ob} and Etayo \cite{ES}, respectively. Moreover, Fr\"olicher gave a way to construct such a canonical connection with respect to which an almost complex structure is parallel, starting from a given one. In particular, he showed that if the initial connection is torsion-free, then the torsion of the obtained connection is proportional to the Nijenhuis tensor of the almost complex structure \cite{eck, frol}.

In 2003, Hitchin \cite{hi} firstly extended the geometry of the tangent bundle $TM$ of a smooth manifold $M$ to the generalized tangent bundle $TM\oplus T^*M$, where $T^*M$ stands for the cotangent bundle of $M$, replacing the Lie bracket by the Courant bracket. This new generalized version is of interest also in theoretical physics.
Recently, the authors of the present paper developed the geometry of statistical structures in this framework,
showing that the quasi-statistical structures (i.e., statistical structures with torsion) fit perfectly into the generalized geometry setting \cite{BN1, AbAn, blanan, nanic}.
They constructed some special connections on the generalized tangent bundle and characterized the integrability of some generalized structures by means of a different bracket from the Courant bracket, defined in \cite{nani} by using an affine connection $\nabla$ on $M$, and called $\nabla$-bracket.

In the present paper, in the generalized geometry framework, we consider some particular generalized structures, namely, almost quaternionic and almost para-quater\-nion\-ic structures and characterize their integrability with respect to the $\nabla$-bracket on
\linebreak
$TM\oplus T^*M$, defined by an affine connection $\nabla$ on $M$.
Also, we provide necessary and sufficient conditions for these structures to be $\hat \nabla$-parallel and $\hat \nabla^*$-parallel, where $\hat \nabla$ is an affine connection on $TM\oplus T^*M$ induced by $\nabla$, and $\hat \nabla^*$ is its generalized dual connection with respect to a bilinear form $\check h$ on $TM\oplus T^*M$ induced by a non-degenerate symmetric or skew-symmetric $(0,2)$-tensor field $h$ on $M$. As main results, we establish the existence of a canonical connection associated to a generalized quaternionic and to a generalized para-quaternionic structure, i.e., a torsion-free generalized affine connection that parallelizes these structures.
We show that, in the quaternionic case, the canonical connection is the generalized Obata connection and that on a quasi-statistical manifold $(M,h,\nabla)$, an integrable $h$-symmetric and $\nabla$-parallel $(1,1)$-tensor field gives rise to a generalized para-quaternionic structure whose canonical connection is precisely $\hat \nabla^*$. Finally we prove that the generalized affine connection that parallelizes certain families of generalized almost complex and almost product structures is preserved.

\section{Preliminaries}

Let $M$ be a smooth manifold, in what follows we shall denote the $C^\infty$ functions on $M$ by $C^\infty(M)$, the smooth sections of $TM$ by $\Gamma^{\infty}(TM)$, of $T^*M$ by $\Gamma^{\infty}(T^*M)$ and of $TM\oplus T^*M$ by $\Gamma^{\infty}(TM\oplus T^*M)$.

\subsection{Norden and para-Norden structures}

\begin{definition} Let $h$ be a pseudo-Riemannian metric on $M$ and let $J:\Gamma^{\infty}(TM) \rightarrow \Gamma^{\infty}(TM)$ be a $h$-symmetric morphism, i.e., $h(JX,Y)=h(X,JY)$, for any $X,Y \in \Gamma^{\infty}(TM)$. Then

(i) $(h,J)$ is called an \textit{almost Norden structure} (and $(M,h,J)$ is called an \textit{almost Norden manifold}) \cite{nor} if $J$ is an almost complex structure, i.e., $J^2=-I$;

(ii) $(h,J)$ is called an \textit{almost para-Norden structure} (and $(M,h,J)$ is called an \textit{almost para-Norden manifold}) \cite{salimov} if $J$ is an almost product structure, i.e., $J^2=I$.

Moreover, if $J$ is integrable, i.e., its Nijenhuis tensor field $N_J$:
\[
N_{J}(X,Y):=[JX,JY]-J[JX, Y]-J[X, JY]+J^{2}[X,Y]
\]
vanishes for any $X,Y\in {\Gamma}^{\infty}(TM)$, we shall drop the adjective \textit{almost}.
\end{definition}

\subsection{Quasi-statistical manifolds}

The notion of quasi-statistical structure was introduced in 2012 by Matsuzoe.

\begin{definition} \cite{ma}
Let $h$ be a non-degenerate $(0,2)$-tensor field and let $\nabla$ be an affine connection on $M$.
Then, $(h,\nabla)$ is called a \textit{quasi-statistical structure} on $M$ (and $(M,h,\nabla)$ is called a \textit{quasi-statistical manifold}) if $d^{\nabla}h=0$, where
\[
(d^{\nabla}h)(X,Y,Z):=(\nabla _X h)(Y,Z)-(\nabla _Y h)(X,Z)+h(T^{\nabla}(X,Y),Z),
\]
for $X, Y, Z \in \Gamma^{\infty} (TM)$,
\[
(\nabla_Xh)(Y,Z):=X(h(Y,Z))-h(\nabla_XY,Z)-h(Y,\nabla_XZ)
\]
is the covariant derivative of $h$ with respect to $\nabla$ and
\pagebreak
\[
T^\nabla(X,Y):=\nabla_XY-\nabla_Y X-[X,Y]
\]
is the torsion operator of ${\nabla}$, for $[\cdot,\cdot]$ the Lie bracket.
\end{definition}

\subsection{Dualistic structures}

In the framework of pseudo-Riemannian geometry, the notion of dual connection was firstly introduced by Amari, which he used in treating statistical inference problems.

\begin{definition} \cite{A1}
Let $(M,h)$ be a pseudo-Riemannian manifold. Two affine connections $\nabla$ and $\nabla^*$ on $M$ are said to be \textit{dual connections}
with respect to $h$ if they satisfy
\[
X(h(Y,Z))=h(\nabla_XY,Z)+h(Y,\nabla^*_XZ),
\]
for any $X,Y,Z\in \Gamma^{\infty}(TM)$. In this case, $(h;\nabla,\nabla^*)$ is called a \textit{dualistic structure} on $M$.
\end{definition}

Also, the notion of dual connection can be stated with respect to a non-degenerate $(0,2)$-tensor field $h$ on $M$.

\bigskip

Direct computations provide the following expression for the dual connection.

\begin{lemma} \label{l5} \cite{blanan}
If $\nabla$ is an affine connection and $h$ is a non-degenerate symmetric or skew-symmetric $(0,2)$-tensor field on $M$, then, the dual connection $\nabla^*$ of $\nabla$ with respect to $h$ is given by:
\[
\nabla^*_XY=\nabla_XY+h^{-1}((\nabla_Xh)(Y)),
\]
\[
\nabla^*_X\beta=\nabla_X\beta-(\nabla_Xh)(h^{-1}(\beta)),
\]
for any $X,Y\in \Gamma^{\infty}(TM)$ and $\beta\in \Gamma^{\infty}(T^*M)$. Moreover, $(\nabla^*)^*=\nabla$.
\end{lemma}

\subsection{The generalized tangent bundle $TM\oplus T^*M$}

\subsubsection{Some geometrical structures}
Let $h$ be a non-degenerate $(0,2)$-tensor field and let $\nabla$ be an affine connection on $M$.
In \cite{blanan} we defined the bilinear form $\check h$ and the affine connection $\hat{\nabla}$ on the generalized tangent bundle $TM \oplus T^*M$ by:
\[
\check h(X+\eta, Y+\beta):=h(X,Y)+h(h^{-1}(\eta),h^{-1}(\beta)),
\]
\[
\hat{\nabla}_{X+\eta}(Y+\beta):=\nabla_XY+h(\nabla_X(h^{-1}(\beta))),
\]
for $X,Y \in {\Gamma}^{\infty}(TM)$ and $\eta, \beta \in {\Gamma}^{\infty}(T^*M).$ Then
\pagebreak
\[
T^{\hat{\nabla}}(X+\eta, Y+\beta)=T^{\nabla}(X,Y)+h\Big((\nabla_Xh^{-1})\beta-(\nabla_Yh^{-1})\eta\Big).
\]

\bigskip

A direct computation gives the expression of the \textit{generalized dual connection} of $\hat{\nabla}$ with respect to $\check h$.

\begin{proposition} \label{p2} \cite{blanan}
If $\nabla$ is an affine connection and $h$ is a non-degenerate symmetric or skew-symmetric $(0,2)$-tensor field on $M$, then, the dual connection ${\hat \nabla}^*$ of
$\hat \nabla$ with respect to $\check h$, defined by:
\[
\check h(Y+\beta, {\hat {\nabla}}^*_{X+\eta}(Z+\gamma))=X(\check h(Y+\beta,Z+\gamma))-\check h({\hat \nabla}_{X+\eta}(Y+\beta),Z+\gamma),
\]
for $X,Y,Z \in {\Gamma}^{\infty}(TM)$ and $\eta, \beta, \gamma \in {\Gamma}^{\infty}(T^*M)$, is given by:
\begin{equation} \label{22}
{\hat {\nabla}}^*_{X+\eta}(Z+\gamma)=h^{-1}({\nabla}_X(h(Z)))+{\nabla}_X \gamma.
\end{equation}
\end{proposition}
Then
\begin{equation*} \label{j}
T^{{\hat {\nabla}}^*}(X+\eta, Y+\beta)=h^{-1}\Big((\nabla_Xh)Y-(\nabla_Yh)X+h(T^{\nabla}(X,Y))\Big).
\end{equation*}
In particular $T^{{\hat {\nabla}}^*}=0$ if and only if $(M,h,\nabla)$ is a quasi-statistical manifold.\\

For the dualistic structure $(\check{h}; \hat{\nabla},\hat{\nabla}^*)$ considered above, we defined a family of connections, $\{ \hat{\nabla}^{(\alpha)}\}$, on $TM\oplus T^*M$, for any $\alpha \in \mathbb{R}$, called \textit{$\alpha$-connections} \cite{BN1}:
\[
\hat{\nabla}^{(\alpha)}:=\frac{1+\alpha}{2}\hat{\nabla}+\frac{1-\alpha}{2}\hat{\nabla}^*.
\]

\subsubsection{Generalized structures}

Next we shall recall the definitions of some well known generalized structures.

\begin{definition}
A morphism $\hat J:\Gamma^{\infty}(TM\oplus T^*M)\rightarrow \Gamma^{\infty}(TM\oplus T^*M)$ is called
a \textit{generalized almost complex structure on $M$} (respectively, \textit{generalized almost product structure on $M$}), if $\hat J^2=-I$ (respectively, $\hat J^2=I$).
\end{definition}

\begin{definition}
(i) A triple $(\hat J_1,\hat J_2,\hat J_3)$ of generalized almost complex structures $\hat J_1$, $\hat J_2$ and $\hat J_3$ is called a \textit{generalized almost quaternionic, or almost hypercomplex, structure on $M$} if $\hat J_1\hat J_2=-\hat J_2\hat J_1$ and $\hat J_1\hat J_2=\hat J_3$.

(ii) A triple $(\hat J_1,\hat J_2,\hat J_3)$ of a generalized almost complex structure $\hat J_1$ and two generalized almost product structures $\hat J_2$ and $\hat J_3$ is called a \textit{generalized almost para-quaternionic structure on $M$} if $\hat J_1\hat J_2=-\hat J_2\hat J_1$ and $\hat J_1\hat J_2=\hat J_3$.
\end{definition}

\begin{remark}
(i) If $(\hat J_1,\hat J_2,\hat J_3)$ is a generalized almost quaternionic structure, then $\hat J_1\hat J_3=-\hat J_3\hat J_1$ and $\hat J_2\hat J_3=-\hat J_3\hat J_2$.

(ii) If $(\hat J_1,\hat J_2,\hat J_3)$ is a generalized almost para-quaternionic structure, then $\hat J_1\hat J_3=-\hat J_3\hat J_1$ and $\hat J_2\hat J_3=-\hat J_3\hat J_2$.
\end{remark}

In this paper we will study generalized almost quaternionic and generalized almost para-quaternionic structures. However, for the sake of completeness, we add the definition of the other generalized triple structures, too.

\begin{definition}
(iii) Given a pair $(\hat J_1,\hat J_2)$ of a generalized almost complex structure $\hat J_1$ and a generalized almost product structure $\hat J_2$ such that $\hat J_1\hat J_2=\hat J_2\hat J_1$, the triple $(\hat J_1,\hat J_2, \hat J_3:=\hat J_1 \hat J_2)$ is called a \textit{generalized almost complex-product structure}, or a  \textit{generalized almost bicomplex structure on $M$}.

(iv) A triple $(\hat J_1,\hat J_2,\hat J_3)$ of generalized almost product structures $\hat J_1$, $\hat J_2$ and $\hat J_3$ is called a \textit{generalized almost hyperproduct structure on $M$} if $\hat J_1\hat J_2=\hat J_2\hat J_1$ and $\hat J_1\hat J_2=\hat J_3$.
\end{definition}

\subsubsection{Integrability}

In order to define the integrability of generalized structures, we need an appropriate bracket on $T(TM\oplus T^*M)$.
For a fixed affine connection $\nabla$ on $M$, we consider the bracket $[\cdot,\cdot]_{\nabla}$ defined in \cite{nani}:
\[
[X+\eta,Y+\beta]_{\nabla}:=[X,Y]+\nabla_X\beta-\nabla_Y\eta,
\]
for $X,Y\in {\Gamma}^{\infty}(TM)$ and $\eta,\beta\in {\Gamma}^{\infty}(T^*M)$, and we define the integrability of a generalized structure with respect to this bracket.

\begin{definition}
A generalized almost complex (or almost product, or almost quaternionic, or almost para-quaternionic) structure $\hat{J}$ is called {\it $\nabla$-integrable} if its Nijenhuis tensor field $N_{\hat{J}}^{\nabla}$ with respect to $\nabla$:
\[
N_{\hat{J}}^{\nabla}(\sigma, \tau):=[\hat{J}\sigma,\hat{J}\tau]_{\nabla}-\hat{J}[\hat{J}\sigma, \tau]_{\nabla}-\hat{J}[\sigma, \hat{J}\tau]_{\nabla} +\hat{J}^{2}[\sigma, \tau]_{\nabla}
\]
vanishes for any $\sigma,\tau \in {\Gamma}^{\infty}(TM\oplus T^*M)$. In this case, we shall drop the adjective \textit{almost}.
\end{definition}

For generalized almost quaternionic and generalized almost para-quaternionic structures on $M$ we have the following.
\begin{proposition}\label{i} Let $(\hat J_1,\hat J_2,\hat J_3)$ be a generalized almost quaternionic structure on $M$. If $\hat J_1$ and $\hat J_2$ are $\nabla$-integrable, then $\hat J_3$ is $\nabla$-integrable, too.
\end{proposition}

\begin{proof} Indeed, from $\hat J_1^2=\hat J_2^2=\hat J_3^2=-I$, $\hat J_3=\hat J_1 \hat J_2$, $\hat J_1 \hat J_2=-\hat J_2 \hat J_1$, a direct computation gives the following formula:
$$2N_{\hat{J_3}}^{\nabla}(\sigma, \tau)=N_{\hat{J_1}}^{\nabla}(\hat J_2\sigma, \hat J_2\tau)+N_{\hat{J_2}}^{\nabla}(\hat J_1\sigma, \hat J_1\tau)-\hat J_1N_{\hat{J_2}}^{\nabla}(\hat J_1\sigma, \tau)-\hat J_1 N_{\hat{J_2}}^{\nabla}(\sigma, \hat J_1\tau)+$$
$$+N_{\hat{J_2}}^{\nabla}(\sigma, \tau)-\hat J_2N_{\hat{J_1}}^{\nabla}(\hat J_2\sigma,\tau)-\hat J_2 N_{\hat{J_1}}^{\nabla}(\sigma, \hat J_2\tau)+N_{\hat{J_1}}^{\nabla}(\sigma, \tau),$$
for any $\sigma, \tau \in \Gamma^\infty(TM\oplus T^*M)$, then the statement.
\end{proof}

\begin{proposition}\label{ii} Let $(\hat J_1,\hat J_2,\hat J_3)$ be a generalized almost para-quaternionic structure on $M$. If $\hat J_1$ and $\hat J_2$ are $\nabla$-integrable, then $\hat J_3$ is $\nabla$-integrable, too. Moreover, if $\hat J_2$ and $\hat J_3$ are $\nabla$-integrable, then $\hat J_1$ is $\nabla$-integrable, too.
\end{proposition}

\begin{proof} Indeed, from $\hat J_1^2=-I$, $\hat J_2^2=\hat J_3^2=I$, $\hat J_3=\hat J_1 \hat J_2$, $\hat J_1 \hat J_2=-\hat J_2 \hat J_1$, a direct computation gives the following formulas:
$$2N_{\hat{J_3}}^{\nabla}(\sigma, \tau)=N_{\hat{J_1}}^{\nabla}(\hat J_2\sigma, \hat J_2\tau)+N_{\hat{J_2}}^{\nabla}(\hat J_1\sigma, \hat J_1\tau)-\hat J_1N_{\hat{J_2}}^{\nabla}(\hat J_1\sigma, \tau)-\hat J_1 N_{\hat{J_2}}^{\nabla}(\sigma, \hat J_1\tau)+$$
$$+N_{\hat{J_2}}^{\nabla}(\sigma, \tau)-\hat J_2N_{\hat{J_1}}^{\nabla}(\hat J_2\sigma,\tau)-\hat J_2 N_{\hat{J_1}}^{\nabla}(\sigma, \hat J_2\tau)-N_{\hat{J_1}}^{\nabla}(\sigma, \tau)$$
and
$$2N_{\hat{J_1}}^{\nabla}(\sigma, \tau)=N_{\hat{J_2}}^{\nabla}(\hat J_3\sigma, \hat J_3\tau)+N_{\hat{J_3}}^{\nabla}(\hat J_2\sigma, \hat J_2\tau)-\hat J_2N_{\hat{J_3}}^{\nabla}(\hat J_2\sigma, \tau)-\hat J_2 N_{\hat{J_3}}^{\nabla}(\sigma, \hat J_2\tau)-$$
$$-N_{\hat{J_3}}^{\nabla}(\sigma, \tau)-\hat J_3N_{\hat{J_2}}^{\nabla}(\hat J_3\sigma,\tau)-\hat J_3 N_{\hat{J_2}}^{\nabla}(\sigma, \hat J_3\tau)-N_{\hat{J_2}}^{\nabla}(\sigma, \tau),$$
for any $\sigma, \tau \in \Gamma^\infty(TM\oplus T^*M)$, then the statement.
\end{proof}

\section{Some generalized almost quaternionic and almost
para-quaternionic structures}

\subsection{Generalized almost para-quaternionic structures defined by commuting $h$-symmetric $(1,1)$-tensors fields}

Let $h$ be a non-degenerate symmetric $(0,2)$-tensor field on a smooth manifold $M$. We shall consider
generalized almost para-quaternionic structures induced by $h$-symmetric $(1,1)$-tensor fields, $J_1$ and $J_2$ on $M$, which commute.

Let
\begin{equation} \label{oo}
\hat{J}_-:=\left(
                    \begin{array}{cc}
                      J_1 & -(J_1^2+I)h^{-1} \\
                      h & -J_1^* \\
                    \end{array}
                  \right), \ \ \hat{J}_+:=\left(
                    \begin{array}{cc}
                      J_2 & -(J_2^2-I)h^{-1} \\
                      h & -J_2^* \\
                    \end{array}
                  \right).
\end{equation}

\begin{proposition} \label{1}
Let $h$ be a non-degenerate symmetric $(0,2)$-tensor field and let $J_1$ and $J_2$ be two $h$-symmetric $(1,1)$-tensor fields on $M$ such that
$J_1J_2=J_2J_1$ and $(J_1-J_2)^2=0$. Then $(\hat{J}_-,\hat{J}_+,\hat{J}_-\hat{J}_+)$ is a generalized almost para-quaternionic structure on $M$, where $\hat{J}_-$
is the generalized almost complex structure and $\hat{J}_+$ is the generalized almost product structure given by (\ref{oo}).
\end{proposition}

\begin{proof}
We have
\[
\hat{J}_-\hat{J}_+=\left(
                    \begin{array}{cc}
                      J_1J_2-(J_1^2+I) & -J_1(J_2^2-I)h^{-1}+(J_1^2+I)h^{-1}J_2^* \\
                      h J_2-J_1^* h & -h (J_2^2-I)h^{-1}+J_1^*J_2^* \\
                    \end{array}
                  \right)
\]
and
\[
\hat{J}_+\hat{J}_-=\left(
                    \begin{array}{cc}
                      J_2J_1-(J_2^2-I) & -J_2(J_1^2+I)h^{-1}+(J_2^2-I)h^{-1}J_1^* \\
                      h J_1-J_2^* h & -h (J_1^2+I)h^{-1}+J_2^*J_1^* \\
                    \end{array}
                  \right)
\]
and taking into account that
\[
h J_i=J_i^* h, \ \ J_ih^{-1}=h^{-1}J_i^*, \ i=1,2, \ \ (J_1J_2)^*=J_2^*J_1^*,
\]
we get $\hat J_-^2=-I$, $\hat J_+^2=I$, $\hat{J}_-\hat{J}_+=-\hat{J}_+\hat{J}_-$,
hence the conclusion.
\end{proof}

For $J_1=J_2$, we have
\begin{equation} \label{op1}
\hat{J}_-:=\left(
                    \begin{array}{cc}
                      J & -(J^2+I)h^{-1} \\
                      h & -J^* \\
                    \end{array}
                  \right), \ \ \hat{J}_+:=\left(
                    \begin{array}{cc}
                      J & -(J^2-I)h^{-1} \\
                      h & -J^* \\
                    \end{array}
                  \right)
\end{equation}
and we can state the following.

\begin{corollary} \label{v}
If $h$ is a non-degenerate symmetric $(0,2)$-tensor field and $J$ is a $h$-symmetric $(1,1)$-tensor field on $M$,
then $(\hat{J}_-,\hat{J}_+,\hat{J}_-\hat{J}_+)$ is a generalized almost para-quaternionic structure on $M$, where $\hat{J}_-$
is the generalized almost complex structure
and $\hat{J}_+$ is the generalized almost product structure
given by (\ref{op1}).
\end{corollary}

\begin{remark}
If $J$ is an almost Norden structure on $M$, then:
\begin{equation*} \label{op}
\hat{J}_-:=\left(
                    \begin{array}{cc}
                      J & 0 \\
                      h & -J^* \\
                    \end{array}
                  \right), \ \ \hat{J}_+:=\left(
                    \begin{array}{cc}
                      J & 2h^{-1} \\
                      h & -J^* \\
                    \end{array}
                  \right)
\end{equation*}
and if $J$ is an almost para-Norden structure on $M$, then:
\begin{equation*} \label{op}
\hat{J}_-:=\left(
                    \begin{array}{cc}
                      J & -2h^{-1} \\
                      h & -J^* \\
                    \end{array}
                  \right), \ \ \hat{J}_+:=\left(
                    \begin{array}{cc}
                      J &0 \\
                      h & -J^* \\
                    \end{array}
                  \right).
\end{equation*}
\end{remark}

Denote by
\begin{equation} \label{jj}
\hat J=:\hat{J}_-\hat{J}_+=\left(
                    \begin{array}{cc}
                      -I & 2Jh^{-1} \\
                      0 & I \\
                    \end{array}
                  \right).
                  \end{equation}
We can state the following results.

\begin{proposition} \label{mm3}
Let $\nabla$ be an affine connection, let $h$ be a non-degenerate symmetric $(0,2)$-tensor field, let $J$ be a $h$-symmetric $(1,1)$-tensor field on $M$, and let $\hat J$ be given by (\ref{jj}). Then the following assertions are equivalent:

(i) $\hat{\nabla}\hat{J}=0$;

(ii) $\nabla J=0$;

(iii) $\hat{\nabla}^*\hat{J}=0$.
\end{proposition}

\begin{proof}
For any $X,Y\in {\Gamma}^{\infty}(TM)$ and $\eta,\beta \in {\Gamma}^{\infty}(T^*M)$, we have
$$(\hat{\nabla}_{X+\eta}\hat{J})(Y+\beta):=\hat{\nabla}_{X+\eta}(\hat{J}(Y+\beta))-\hat{J}(\hat{\nabla}_{X+\eta}(Y+\beta))=$$
$$=\hat{\nabla}_{X+\eta}\Big(-Y+2J(h^{-1}(\beta))+\beta\Big)-\hat J\Big(\nabla_XY+h(\nabla_X(h^{-1}(\beta)))\Big)=$$
$$=\nabla_X\Big(-Y+2J(h^{-1}(\beta))\Big)+h(\nabla_X(h^{-1}(\beta)))-$$
$$-\Big(-\nabla_XY+2J(\nabla_X(h^{-1}(\beta)))+h(\nabla_X(h^{-1}(\beta)))\Big)=2(\nabla_XJ)(h^{-1}(\beta))$$
and we get (i) $\Longleftrightarrow$ (ii). Also:
$$(\hat{\nabla}^*_{X+\eta}\hat{J})(Y+\beta):=\hat{\nabla}^*_{X+\eta}(\hat{J}(Y+\beta))-\hat{J}(\hat{\nabla}^*_{X+\eta}(Y+\beta))=$$
$$=\hat{\nabla}^*_{X+\eta}\Big(-Y+2J(h^{-1}(\beta))+\beta\Big)-\hat J\Big(h^{-1}(\nabla_X(h(Y)))+\nabla_X\beta\Big)=$$
$$=h^{-1}(\nabla_X(h(-Y+2J(h^{-1}(\beta)))))+\nabla_X\beta-\Big(-h^{-1}(\nabla_X(h(Y)))+2J(h^{-1}(\nabla_X\beta))+\nabla_X\beta\Big)=$$
$$=2\Big(h^{-1}(\nabla_X(h(J(h^{-1}(\beta)))))-J(h^{-1}(\nabla_X\beta))\Big)=$$
$$=2\Big(h^{-1}(\nabla_X J^*\beta)-h^{-1}(J^*(\nabla_X\beta))\Big)=2h^{-1}((\nabla_XJ^*)\beta)$$
and (ii) $\Longleftrightarrow$ (iii) follows since $((\nabla_XJ^*)\beta)(Y)=\beta((\nabla_XJ)Y)$, for any $X,Y\in {\Gamma}^{\infty}(TM)$ and $\beta \in {\Gamma}^{\infty}(T^*M)$.
\end{proof}

\begin{corollary}
Let $\nabla$ be an affine connection, let $h$ be a non-degenerate symmetric $(0,2)$-tensor field, let $J$ be a $h$-symmetric $(1,1)$-tensor field on $M$, and let $\hat J$ be given by (\ref{jj}).
Then $\hat{\nabla}^{(\alpha)}\hat{J}=0$, for any $\alpha\in \mathbb{R}$, if and only if $\nabla J=0$.
\end{corollary}

\begin{proof}
Since for any $X,Y\in {\Gamma}^{\infty}(TM)$ and $\eta,\beta \in {\Gamma}^{\infty}(T^*M)$, we have
\[
(\hat{\nabla}^{(\alpha)}_{X+\eta}\hat{J})(Y+\beta)=\frac{1+\alpha}{2}(\hat{\nabla}_{X+\eta}\hat{J})(Y+\beta)+
\frac{1-\alpha}{2}(\hat{\nabla}^*_{X+\eta}\hat{J})(Y+\beta),
\]
we get the conclusion from Proposition \ref{mm3}.
\end{proof}

\begin{proposition}\label{dddd}
Let $\nabla$ be an affine connection, let $h$ be a non-degenerate symmetric $(0,2)$-tensor field, and let $J$ be a $h$-symmetric $(1,1)$-tensor field on $M$.
Then the generalized almost product structure
$\hat J$ given by (\ref{jj})
is $\nabla$-integrable if and only if
\[
(\nabla_{JX}h)(JY)-(\nabla_{JY}h)(JX)+h(T^{\nabla}(JX,JY))=(\nabla_{JX}J^*)(h(Y))-(\nabla_{JY}J^*)(h(X)),
\]
for any $X,Y\in {\Gamma}^{\infty}(TM)$.
\end{proposition}

\begin{proof}
For any $X,Y\in {\Gamma}^{\infty}(TM)$, we have
\[
N_{\hat{J}}^{\nabla}(X,Y)=0, \ \ N_{\hat{J}}^{\nabla}(X,h(Y))=0,
\]
\[
N_{\hat{J}}^{\nabla}(h(X),h(Y))=4\Big([JX,JY]-J\big(h^{-1}(\nabla_{JX}(h(Y))-\nabla_{JY}(h(X)))\big)\Big).
\]

Taking into account that
\[
J\big(h^{-1}(\nabla_{JX}(h(Y)))\big)=h^{-1}(\nabla_{JX}(h(JY)))-h^{-1}((\nabla_{JX}J^*)(h(Y)))=
\]
$$=h^{-1}((\nabla_{JX}h)(JY))+\nabla_{JX}JY-h^{-1}((\nabla_{JX}J^*)(h(Y))),$$
we infer
$$N_{\hat{J}}^{\nabla}(h(X),h(Y))=4\Big(h^{-1}\big((\nabla_{JX}J^*)(h(Y))-(\nabla_{JY}J^*)(h(X))\big)-$$$$-
h^{-1}\big((\nabla_{JX}h)(JY)-(\nabla_{JY}h)(JX)-h(T^{\nabla}(JX,JY))\big)\Big).$$
Then the proof is complete.
\end{proof}

By means of Proposition \ref{dddd} we can characterize $\nabla$-integrability of $\hat J$ in terms of quasi-statistical structures.

\begin{lemma}
Let $(M,h,\nabla)$ be a quasi-statistical manifold and let $J$ be a $h$-symmetric $(1,1)$-tensor field on $M$. Then
the generalized almost product structure $\hat J$ given by (\ref{jj}) is $\nabla$-integrable if and only if
\[
(\nabla_{JX}J^*)(h(Y))=(\nabla_{JY}J^*)(h(X)),
\]
for any $X,Y\in {\Gamma}^{\infty}(TM)$.
\end{lemma}

\begin{proposition}
Let $\nabla$ be an affine connection, let $h$ be a non-degenerate symmetric $(0,2)$-tensor field, and let $J$ be an invertible $\nabla$-parallel $h$-symmetric $(1,1)$-tensor field on $M$.
Then the generalized almost product structure $\hat J$ given by (\ref{jj}) is $\nabla$-integrable if and only if
$(h,\nabla)$ is a quasi-statistical structure on $M$.
\end{proposition}

On the other hand, we have the following.
\begin{proposition} Let $(M,h,\nabla)$ be a quasi-statistical manifold and let $J$ be a $h$-symmetric, $\nabla$-parallel $(1,1)$-tensor field on $M$ such that the Nijenhuis tensor of $J$, $N_J$, vanishes. Then
the generalized structures $\hat J_\mp:=\left(
                    \begin{array}{cc}
                      J & -(J^2\pm I)h^{-1} \\
                      h & -J^* \\
                    \end{array}
                  \right)$ are $\nabla$-integrable. Moreover, $\hat \nabla^* \hat J_{\mp}=0$.
\end{proposition}

\begin{proof} A direct computation gives the following:
$$N^\nabla_{\hat J_\mp}(X,Y)=N_J(X,Y)+\Big((\nabla_{JX}h)Y-(\nabla_{Y}h)JX+h(T^\nabla(JX,Y))\Big)+$$
$$+\Big((\nabla_{X}h)JY-(\nabla_{JY}h)X+h(T^\nabla(X,JY))\Big)+(J^2\pm I)h^{-1}\Big((\nabla_{X}h)Y-(\nabla_{Y}h)X+h(T^\nabla(X,Y))\Big),$$
$$N^\nabla_{\hat J_\mp}(h(X),h(Y))=h^{-1}\Big((\nabla_{(J^2 \pm I)Y}h)(J^2\pm I)X-(\nabla_{(J^2\pm I)X}h)(J^2\pm I)Y)+$$
$$+h(T^\nabla((J^2\pm I)Y,(J^2\pm I)X))\Big),$$
$$N^\nabla_{\hat J_\mp}(X,h(Y))=h^{-1}\Big((\nabla_{JX}h)(J^2\pm I)Y-(\nabla_{(J^2\pm I)Y}h)(JX)+h(T^\nabla(JX,(J^2\pm I)Y))\Big)+$$
$$-J\Big(h^{-1}\Big((\nabla_{X}h)(J^2\pm I)Y-(\nabla_{(J^2\pm I)Y}h)X+h(X,T^\nabla((J^2\pm I)Y))\Big)\Big)-$$
$$-\Big((\nabla_{X}h)(J^2\pm I)Y-(\nabla_{(J^2 \pm I)Y}h)X+h(T^\nabla(X,(J^2\pm I)Y))\Big),$$
for any $X,Y\in \Gamma^\infty(TM)$.
Moreover, a direct computation gives $\hat \nabla^* \hat J_{\mp}=0$.
\end{proof}

Based on the previous results, we can state.

\begin{corollary}\label{th} Let $(M,h,\nabla)$ be a quasi-statistical manifold and let $J$ be a $h$-symmetric, $\nabla$-parallel $(1,1)$-tensor field on $M$ such that the Nijenhuis tensor of $J$, $N_J$, vanishes. Then
the generalized structures $\hat J_\mp$ and $\hat J$ given by (\ref{op1}) and (\ref{jj}) are $\nabla$-integrable and $\hat \nabla^* $ is a torsion-free generalized affine connection that parallelizes the generalized para-quaternionic structure
$ (\hat J_-,\hat J_+, \hat J)$.
\end{corollary}

For $J_1=J_2=0$, we have
\begin{equation} \label{or}
\hat{J}_{-}:=\left(
                    \begin{array}{cc}
                      0 & -h^{-1} \\
                      h & 0 \\
                    \end{array}
                  \right), \ \ \\
\hat{J}_{+}:=\left(
                    \begin{array}{cc}
                      0 & h^{-1} \\
                      h & 0 \\
                    \end{array}
                  \right), \ \ \\
\hat J=:\hat{J}_{-}\hat{J}_{+}=\left(
                    \begin{array}{cc}
                      -I & 0 \\
                      0 & I \\
                    \end{array}
                  \right)
\end{equation}
and we can state the following.

\begin{proposition}
Let $\nabla$ be an affine connection, let $h$ be a non-degenerate symmetric $(0,2)$-tensor field, and let $\hat J_{\mp}$ and $\hat J$ be given by (\ref{or}).
Then

(i) $\hat{\nabla}^{(\alpha)}\hat{J}_{\mp}=0$ and $\hat{\nabla}^{(\alpha)}\hat{J}=0$, for any $\alpha\in \mathbb{R}$; in particular, $\hat \nabla \hat J_{\mp}=0$, $\hat \nabla \hat J=0$, $\hat \nabla^* \hat J_{\mp}=0$ and $\hat \nabla^* \hat J=0$;

(ii) the generalized structures
$\hat J_{\mp}$ and $\hat J$ given by (\ref{or}) are $\nabla$-integrable.
\end{proposition}

\begin{proof}
(i) follows from Proposition \ref{mm3} and (ii) follows from Proposition \ref{dddd} and by a direct computation.
\end{proof}

\subsection{Generalized triple structures defined by two $h$-symmetric $(1,1)$-tensors fields and two morphisms from $T^*M$ to $TM$}

For general $h$-symmetric $(1,1)$-tensor fields, $J_1$ and $J_2$, we shall consider the following generalized almost quaternionic and generalized almost para-quaternionic structures.

Let
\begin{equation} \label{os}
\hat{J}_i:=\left(
                    \begin{array}{cc}
                      J_i & H_i \\
                      h & -J_i^* \\
                    \end{array}
                  \right), \ i=1,2,
\end{equation}
where $H_i:\Gamma^{\infty}(T^*M)\rightarrow \Gamma^{\infty}(TM)$ are two morphisms.

\begin{proposition} \label{sdk}
Let $h$ be a non-degenerate symmetric $(0,2)$-tensor field, let $J_i$ be two $h$-symmetric $(1,1)$-tensor fields on $M$, $i=1,2$, and let
$\hat{J}_i$ be given by (\ref{os}), where $H_i:\Gamma^{\infty}(T^*M)\rightarrow \Gamma^{\infty}(TM)$ are two morphisms. Then $\hat{J}_1\hat{J}_2=-\hat{J}_2\hat{J}_1$ if and only if
$$\left\{
    \begin{array}{ll}
      J_1J_2+J_2J_1=-(H_1+H_2)h \\
      J_1^*J_2^*+J_2^*J_1^*=-h(H_1+H_2) \\
      J_1H_2-H_2J_1^*=H_1J_2^*-J_2H_1
    \end{array}
  \right..
$$
\end{proposition}

\begin{proof}
We have
\[
\hat{J}_1\hat{J}_2=\left(
                    \begin{array}{cc}
                      J_1J_2+H_1h & J_1H_2-H_1J_2^* \\
                      h J_2-J_1^* h & hH_2+J_1^*J_2^* \\
                    \end{array}
                  \right)
\]
and
\[
\hat{J}_2\hat{J}_1=\left(
                    \begin{array}{cc}
                      J_2J_1+H_2h & J_2H_1-H_2J_1^* \\
                      h J_1-J_2^* h & hH_1+J_2^*J_1^* \\
                    \end{array}
                  \right)
\]
and taking into account that
\[
h J_i=J_i^* h, \ i=1,2,
\]
we get the conclusion.
\end{proof}

Let $H_i=\lambda_iI$, $ i=1,2$. Then
\begin{equation} \label{om}
\hat{J}_i:=\left(
                    \begin{array}{cc}
                      J_i & \lambda_ih^{-1} \\
                      h & -J_i^* \\
                    \end{array}
                  \right)
, \ \ \lambda_i\in \mathbb R, \ i=1,2,
\end{equation}
hence $$\hat{J}_i^2=\left(
                    \begin{array}{cc}
                      J_i^2 & 0 \\
                      0 & (J_i^*)^2 \\
                    \end{array}
                  \right)+\lambda_i I
$$ and we remark that

(i) $\hat{J}_i$ is a generalized almost complex structure if and only if $J_i^2=-(\lambda_i+1)I$;

(ii) $\hat{J}_i$ is a generalized almost product structure if and only if $J_i^2=-(\lambda_i-1)I$;

(iii) $\hat{J}_1\hat{J}_2=-\hat{J}_2\hat{J}_1$ if and only if $J_1J_2+J_2J_1=-(\lambda_1+\lambda_2)I$.

\begin{proposition}
Let $h$ be a non-degenerate symmetric $(0,2)$-tensor field, let $J_i$ be two $h$-symmetric $(1,1)$-tensor fields on $M$, $i=1,2$, such that $J_i^2=-(\lambda_i+1)I$, for $\lambda_i\in \mathbb R$, and $J_1J_2+J_2J_1=-(\lambda_1+\lambda_2)I$, and let
$\hat{J}_i$ be given by (\ref{om}). Then $(\hat{J}_1,\hat{J}_2,\hat{J}_1\hat{J}_2)$ is a generalized almost quaternionic structure on $M$.
\end{proposition}

\begin{proof}
We have $\hat J_1^2=-I$, $\hat J_2^2=-I$, $\hat{J}_1\hat{J}_2=-\hat{J}_2\hat{J}_1$,
hence the conclusion.
\end{proof}

\begin{proposition}
Let $h$ be a non-degenerate symmetric $(0,2)$-tensor field, let $J_i$ be two $h$-symmetric $(1,1)$-tensor fields on $M$, $i=1,2$, such that $J_1^2=-(\lambda_1+1)I$, $J_2^2=-(\lambda_2-1)I$, for $\lambda_1,\lambda_2\in \mathbb R$, and $J_1J_2+J_2J_1=-(\lambda_1+\lambda_2)I$, and let
$\hat{J}_i$ be given by (\ref{om}). Then $(\hat{J}_1,\hat{J}_2,\hat{J}_1\hat{J}_2)$ is a generalized almost para-quaternionic structure on $M$.
\end{proposition}

\begin{proof}
We have $\hat J_1^2=-I$, $\hat J_2^2=I$, $\hat{J}_1\hat{J}_2=-\hat{J}_2\hat{J}_1$,
hence the conclusion.
\end{proof}

\begin{proposition}
Let $h$ be a non-degenerate symmetric $(0,2)$-tensor field, let $J_i$ be two $h$-symmetric $(1,1)$-tensor fields on $M$, $i=1,2$, such that $J_i^2=-(\lambda_i-1)I$, for $\lambda_i\in \mathbb R$, and $J_1J_2+J_2J_1=-(\lambda_1+\lambda_2)I$, and let
$\hat{J}_i$ be given by (\ref{om}). Then $(\hat{J}_1\hat{J}_2,\hat{J}_1,\hat{J}_2)$ is a generalized almost para-quaternionic structure on $M$.
\end{proposition}

\begin{proof}
We have $\hat J_1^2=I$, $\hat J_2^2=I$, $\hat{J}_1\hat{J}_2=-\hat{J}_2\hat{J}_1$,
hence the conclusion.
\end{proof}
\pagebreak

Denote by
\begin{equation} \label{kl}
\hat J=:\hat{J}_1\hat{J}_2=\left(
                    \begin{array}{cc}
                      J_1J_2+\lambda_1I & (\lambda_2J_1-\lambda_1J_2)h^{-1} \\
                      hJ_2-J_1^*h & J_1^*J_2^*+\lambda_2I \\
                    \end{array}
                  \right).
\end{equation}
We can state the following results.

\begin{proposition} \label{kj}
Let $\nabla$ be an affine connection, let $h$ be a non-degenerate symmetric $(0,2)$-tensor field, let $J_1$ and $J_2$ be two $h$-symmetric $(1,1)$-tensor fields on $M$, and let $\hat J$ be given by (\ref{kl}).
Then the following assertions are equivalent:

(i) $\hat \nabla \hat J=0$;

(ii) $
\left\{
  \begin{array}{ll}
   \lambda_1=\lambda_2 \\
    \nabla J_1=\nabla J_2 \\
    (\nabla J_i)\circ J_j +J_i\circ (\nabla J_j)=0, \ i,j=1,2, i\neq j
  \end{array}
\right.;
$

(iii) $\hat \nabla^* \hat J=0$.

In particular, if $\nabla J_1=0$ and $\nabla J_2=0$, then $\hat \nabla \hat J=0$ and $\hat \nabla^* \hat J=0$.
\end{proposition}

\begin{proof}
For any $X,Y\in {\Gamma}^{\infty}(TM)$ and $\eta,\beta \in {\Gamma}^{\infty}(T^*M)$, we have
$$(\hat{\nabla}_{X+\eta}\hat{J})(Y+\beta):=\hat{\nabla}_{X+\eta}(\hat{J}(Y+\beta))-
\hat{J}(\hat{\nabla}_{X+\eta}(Y+\beta))=$$
$$=\hat{\nabla}_{X+\eta}\Big((J_1J_2+\lambda_1I)(Y)+(\lambda_2J_1-\lambda_1J_2)(h^{-1}(\beta))+h((J_2-J_1)(Y))+(J_1^*J_2^*+\lambda_2I)(\beta)\Big)-
$$$$-
\hat J\Big(\nabla_XY+h(\nabla_X(h^{-1}(\beta)))\Big)=$$
$$=\nabla_X(J_1J_2Y)+\lambda_1\nabla_XY+\lambda_2\nabla_X(J_1(h^{-1}(\beta)))-\lambda_1\nabla_X(J_2(h^{-1}(\beta)))+
$$$$+h\big(\nabla_X(J_2Y-J_1Y+
h^{-1}(J_1^*J_2^*\beta)+\lambda_2h^{-1}(\beta))\big)-$$
$$-\Big((J_1J_2+\lambda_1I)(\nabla_XY)+(\lambda_2J_1-\lambda_1J_2)(\nabla_X(h^{-1}(\beta)))+
$$$$+h((J_2-J_1)(\nabla_XY))+(J_1^*J_2^*+\lambda_2I)(h(\nabla_X(h^{-1}(\beta)))\Big)=$$
$$=(\nabla_X(J_1J_2))Y+\lambda_2(\nabla_XJ_1)(h^{-1}(\beta))-\lambda_1(\nabla_XJ_2)(h^{-1}(\beta))+
$$$$+h\big((\nabla_XJ_2)Y-(\nabla_XJ_1)Y\big)+h(\nabla_X(h^{-1}(J_1^*J_2^*\beta)))-J_1^*J_2^*h(\nabla_X(h^{-1}(\beta))).$$

Taking into account that
\[
J_1^*J_2^*h=hJ_2J_1
\]
(and, consequently, $h^{-1}J_1^*J_2^*=J_2J_1h^{-1}$), we get
\[
h(\nabla_X(h^{-1}(J_1^*J_2^*\beta)))-J_1^*J_2^*h(\nabla_X(h^{-1}(\beta)))=$$
$$=h(\nabla_X((J_2J_1)(h^{-1}(\beta))))-h((J_2J_1)(\nabla_X(h^{-1}(\beta))))
=h\big((\nabla_X(J_2J_1))(h^{-1}(\beta))\big),
\]
\pagebreak
hence
\[
(\hat{\nabla}_{X+\eta}\hat{J})(Y+\beta)=(\nabla_X(J_1J_2))Y+\lambda_2(\nabla_XJ_1)(h^{-1}(\beta))-\lambda_1(\nabla_XJ_2)(h^{-1}(\beta))+
\]
\[
+h\big((\nabla_XJ_2)Y-(\nabla_XJ_1)Y+(\nabla_X(J_2J_1))(h^{-1}(\beta))\big).
\]

Also:
\[
(\nabla_X(J_1J_2))Y=(\nabla_XJ_1)(J_2Y)+J_1((\nabla_XJ_2)Y)
\]
and we obtain
\[
(\hat{\nabla}_{X+\eta}\hat{J})(Y+\beta)=(\nabla_XJ_1)(J_2Y)+J_1((\nabla_XJ_2)Y)+
\]
\[
+\lambda_2(\nabla_XJ_1)(h^{-1}(\beta))-\lambda_1(\nabla_XJ_2)(h^{-1}(\beta))+
\]
\[
+h\big((\nabla_XJ_2)Y-(\nabla_XJ_1)Y+(\nabla_XJ_2)(J_1(h^{-1}(\beta)))+J_2((\nabla_XJ_1)(h^{-1}(\beta)))\big).
\]

Therefore, $\hat \nabla \hat J=0$ if and only if
\[
\left\{
  \begin{array}{ll}
    (\nabla_XJ_1)(J_2Y)+J_1((\nabla_XJ_2)Y)+\lambda_2(\nabla_XJ_1)(h^{-1}(\beta))-\lambda_1(\nabla_XJ_2)(h^{-1}(\beta))=0 \\
    (\nabla_XJ_2)Y-(\nabla_XJ_1)Y+(\nabla_XJ_2)(J_1(h^{-1}(\beta)))+J_2((\nabla_XJ_1)(h^{-1}(\beta)))=0
  \end{array}
\right.,
\]
for any $X,Y\in \Gamma^{\infty}(TM)$ and $\beta \in \Gamma^{\infty}(T^*M)$. Taking, consequently, $\beta=0$ and $Y=0$, we conclude that
$\hat \nabla \hat J=0$ if and only if
\[
\left\{
  \begin{array}{ll}
    (\nabla_XJ_1)(J_2Y)+J_1((\nabla_XJ_2)Y)=0 \\
    (\nabla_XJ_2)Y-(\nabla_XJ_1)Y=0\\
  \lambda_2(\nabla_XJ_1)(h^{-1}(\beta))-\lambda_1(\nabla_XJ_2)(h^{-1}(\beta))=0\\
  (\nabla_XJ_2)(J_1(h^{-1}(\beta)))+J_2((\nabla_XJ_1)(h^{-1}(\beta)))=0
\end{array}
\right.,
\]
for any $X,Y\in \Gamma^{\infty}(TM)$ and $\beta\in \Gamma^{\infty}(T^*M)$, i.e.,
\pagebreak
\[
\left\{
  \begin{array}{ll}
    (\nabla J_1)\circ J_2+J_1\circ (\nabla J_2)=0 \\
    \nabla J_2=\nabla J_1\\
\lambda_2\nabla J_1=\lambda_1\nabla J_2\\
(\nabla J_2)\circ J_1+J_2\circ (\nabla J_1)=0
  \end{array}
\right.
\]
and we get (i) $\Longleftrightarrow$ (ii).

Now
$$(\hat{\nabla}^*_{X+\eta}\hat{J})(Y+\beta):=\hat{\nabla}^*_{X+\eta}(\hat{J}(Y+\beta))-\hat{J}(\hat{\nabla}^*_{X+\eta}(Y+\beta))=$$
$$=\hat{\nabla}^*_{X+\eta}\Big((J_1J_2+\lambda_1I)(Y)+(\lambda_2J_1-\lambda_1J_2)(h^{-1}(\beta))+h((J_2-J_1)(Y))+(J_1^*J_2^*+\lambda_2I)(\beta)\Big)-
$$$$-\hat J\Big(h^{-1}(\nabla_X(h(Y)))+\nabla_X\beta\Big)=$$
$$=h^{-1}(\nabla_X(h(J_1J_2Y)))+\lambda_1h^{-1}(\nabla_X(h(Y)))+$$$$+\lambda_2h^{-1}(\nabla_X(h(J_1(h^{-1}(\beta)))))-\lambda_1h^{-1}(\nabla_X(h(J_2(h^{-1}(\beta)))))   +$$$$+\nabla_X\big(h(J_2Y)-h(J_1Y)+J_1^*J_2^*\beta+\lambda_2\beta\big)-$$$$-\Big(J_1J_2(h^{-1}(\nabla_X(h(Y))))+\lambda_1h^{-1}(\nabla_X(h(Y)))
+\lambda_2J_1(h^{-1}(\nabla_X\beta))-\lambda_1J_2(h^{-1}(\nabla_X\beta))+
$$$$+h(J_2(h^{-1}(\nabla_X(h(Y)))))-h(J_1(h^{-1}(\nabla_X(h(Y)))))+J_1^*J_2^*(\nabla_X\beta)+\lambda_2\nabla_X\beta\Big).$$

Taking into account that
\[
J_1^*J_2^*h=hJ_2J_1, \ \ J_i^*h=hJ_i, \ i=1,2
\]
(and, consequently, $h^{-1}J_1^*J_2^*=J_2J_1h^{-1}$ and $h^{-1}J_i^*=J_ih^{-1}$, $i=1,2$), we get
\[
(\hat{\nabla}^*_{X+\eta}\hat{J})(Y+\beta)=h^{-1}\big((\nabla_X(J_2^*J_1^*))(h(Y))+\lambda_2(\nabla_XJ_1^*)\beta-\lambda_1(\nabla_XJ_2^*)\beta\big)+
\]
\[
+(\nabla_XJ_2^*)(h(Y))-(\nabla_XJ_1^*)(h(Y))+(\nabla_X(J_1^*J_2^*))\beta.
\]

Also:
\[
(\nabla_X(J_1^*J_2^*))\beta=(\nabla_XJ_1^*)(J_2^*\beta)+J_1^*((\nabla_XJ_2^*)\beta)
\]
and we obtain
\[
(\hat{\nabla}^*_{X+\eta}\hat{J})(Y+\beta)=h^{-1}\big((\nabla_XJ_2^*)(J_1^*(h(Y)))+J_2^*((\nabla_XJ_1^*)(h(Y)))+\lambda_2(\nabla_XJ_1^*)\beta-\lambda_1(\nabla_XJ_2^*)\beta\big)+
\]
\[
+(\nabla_XJ_2^*)(h(Y))-(\nabla_XJ_1^*)(h(Y))+(\nabla_XJ_1^*)(J_2^*\beta)+J_1^*((\nabla_XJ_2^*)\beta).
\]

Therefore, $\hat \nabla^* \hat J=0$ if and only if
\[
\left\{
  \begin{array}{ll}
    (\nabla_XJ_2^*)(J_1^*(h(Y)))+J_2^*((\nabla_XJ_1^*)(h(Y)))+\lambda_2(\nabla_XJ_1^*)\beta-\lambda_1(\nabla_XJ_2^*)\beta=0 \\
    (\nabla_XJ_2^*)(h(Y))-(\nabla_XJ_1^*)(h(Y))+(\nabla_XJ_1^*)(J_2^*\beta)+J_1^*((\nabla_XJ_2^*)\beta)=0
  \end{array}
\right.,
\]
for any $X,Y\in \Gamma^{\infty}(TM)$ and $\beta \in \Gamma^{\infty}(T^*M)$. Taking, consequently, $\beta=0$ and $Y=0$, we conclude that
$\hat \nabla^* \hat J=0$ if and only if
\[
\left\{
  \begin{array}{ll}
    (\nabla_XJ_2^*)(J_1^*(h(Y)))+J_2^*((\nabla_XJ_1^*)(h(Y)))=0 \\
    (\nabla_XJ_2^*)(h(Y))-(\nabla_XJ_1^*)(h(Y))=0\\
\lambda_2(\nabla_XJ_1^*)\beta-\lambda_1(\nabla_XJ_2^*)\beta=0\\
(\nabla_XJ_1^*)(J_2^*\beta)+J_1^*((\nabla_XJ_2^*)\beta)=0
  \end{array}
\right.,
\]
for any $X,Y\in \Gamma^{\infty}(TM)$ and $\beta\in \Gamma^{\infty}(T^*M)$, i.e.,
\[
\left\{
  \begin{array}{ll}
    (\nabla J_2^*)\circ J_1^*+J_2^*\circ (\nabla J_1^*)=0 \\
    \nabla J_2^*=\nabla J_1^*\\
\lambda_2\nabla J_1^*=\lambda_1\nabla J_2^*\\
(\nabla J_1^*)\circ J_2^*+J_1^*\circ (\nabla J_2^*)=0
  \end{array}
\right.
\]
and (ii) $\Longleftrightarrow$ (iii) follows since $((\nabla_X J_i^*)\beta)(Y)=\beta((\nabla_XJ_i)Y)$, for any $X,Y\in \Gamma^{\infty}(TM)$ and $\beta \in \Gamma^{\infty}(T^*M)$, $i=1,2$.
\end{proof}

\begin{corollary} \label{c}
Let $\nabla$ be an affine connection, let $h$ be a non-degenerate symmetric $(0,2)$-tensor field, let $J_1$ and $J_2$ be two $h$-symmetric $(1,1)$-tensor fields on $M$, and let $\hat J$ be given by (\ref{kl}).
Then $\hat{\nabla}^{(\alpha)}\hat{J}=0$, for any $\alpha\in \mathbb{R}$, if and only if
\[
\left\{
  \begin{array}{ll}
   \lambda_1=\lambda_2 \\
    \nabla J_1=\nabla J_2 \\
    (\nabla J_i)\circ J_j +J_i\circ (\nabla J_j)=0, \ i,j=1,2, i\neq j
  \end{array}
\right..
\]

In particular, if $\nabla J_1=0$ and $\nabla J_2=0$, then $\hat{\nabla}^{(\alpha)}\hat{J}=0$, for any $\alpha\in \mathbb{R}$.
\end{corollary}

\begin{proof}
Since we have
$$\hat{\nabla}^{(\alpha)}\hat{J}=\frac{1+\alpha}{2}\hat{\nabla}\hat{J}+
\frac{1-\alpha}{2}\hat{\nabla}^*\hat{J},$$
we get the conclusion from Proposition \ref{kj}.
\end{proof}

Moreover, for $H_1=H_2=0$, we consider
\begin{equation} \label{on}
\hat J_i:=\left(
                    \begin{array}{cc}
                      J_i & 0 \\
                      h & -J_i^* \\
                    \end{array}
                  \right), \ i=1,2, \ \ \hat J:=\left(
                    \begin{array}{cc}
                      J_1J_2 & 0 \\
                      h(J_2-J_1) & J_1^*J_2^* \\
                    \end{array}
                  \right)
\end{equation}
and we can state the following.

\begin{proposition} \label{k}
Let $\nabla$ be an affine connection, let $h$ be a non-degenerate symmetric $(0,2)$-tensor field, let $J_1$ and $J_2$ be two $h$-symmetric $(1,1)$-tensor fields on $M$, and let
$\hat J$ be given by (\ref{on}).

(i) If $J_1J_2=J_2J_1$, then $\hat J^2=\left(
                    \begin{array}{cc}
                      J_1^2J_2^2 & 0 \\
                      0 & (J_1^*)^2(J_2^*)^2 \\
                    \end{array}
                  \right)$.

(ii) The following assertions are equivalent:

\hspace*{0.5cm} (ii$_1$) $\hat \nabla \hat J=0$;

\hspace*{0.5cm} (ii$_2$) $\left\{
  \begin{array}{ll}
    \nabla J_1=\nabla J_2 \\
    (\nabla J_i)\circ J_j +J_i\circ (\nabla J_j)=0, \ i,j=1,2, i\neq j
  \end{array}
\right.;
$

\hspace*{0.5cm} (ii$_3$) $\hat \nabla^* \hat J=0$.

In particular, if $\nabla J_1=0$ and $\nabla J_2=0$, then $\hat \nabla \hat J=0$ and $\hat \nabla^* \hat J=0$.
\end{proposition}

\begin{proof}
We have
\[
\hat J^2=\left(
                    \begin{array}{cc}
                      J_1J_2J_1J_2 & 0 \\
                      h(J_2-J_1)J_1J_2+J_1^*J_2^*h(J_2-J_1) & J_1^*J_2^*J_1^*J_2^* \\
                    \end{array}
                  \right).
\]

Taking into account that
\[
J_1^*J_2^*hJ_1=hJ_2J_1^2, \ \ J_1^*J_2^*hJ_2=hJ_1J_2^2,
\]
we get (i). Then (ii) follows from Proposition \ref{kj}.
\end{proof}

From Corollary \ref{c}, we get

\begin{corollary}
Let $\nabla$ be an affine connection, let $h$ be a non-degenerate symmetric $(0,2)$-tensor field, let $J_1$ and $J_2$ be two $h$-symmetric $(1,1)$-tensor fields on $M$, and let $\hat J$ given by (\ref{on}).
Then $\hat{\nabla}^{(\alpha)}\hat{J}=0$, for any $\alpha\in \mathbb{R}$, if and only if
\[
\left\{
  \begin{array}{ll}
    \nabla J_1=\nabla J_2 \\
    (\nabla J_i)\circ J_j +J_i\circ (\nabla J_j)=0, \ i,j=1,2, i\neq j
  \end{array}
\right..
\]

In particular, if $\nabla J_1=0$ and $\nabla J_2=0$, then $\hat \nabla^{(\alpha)} \hat J=0$, for any $\alpha\in \mathbb{R}$.
\end{corollary}

\begin{remark}
Under the assumptions of Proposition \ref{k}, if $J_1^2J_2^2=-I$, then $\hat J$ is a generalized almost complex structure,
and if $J_1^2J_2^2=I$, then $\hat J$ is a generalized almost product structure.
\end{remark}

\subsection{Generalized triple structures via two anticommuting almost complex or almost product structures}

Further, we shall consider generalized almost quaternionic and generalized almost para-quaternionic structures obtained by means of two almost complex and almost product structures, $J_1$ and $J_2$, which anticommute.

Let $J_1$ and $J_2$ be two anticommuting $(1,1)$-tensor fields and let
\begin{equation} \label{n}
\hat{J}_i:=\left(
                    \begin{array}{cc}
                      J_i & 0 \\
                      0 & -J_i^* \\
                    \end{array}
                  \right), \ i=1,2.
\end{equation}

\begin{proposition}
Let $J_1$ and $J_2$ be two almost complex structures on $M$ such that $J_1J_2=-J_2J_1$. Then $(\hat{J}_1,\hat{J}_2,\hat{J}_1\hat{J}_2)$ is a generalized almost quaternionic structure on $M$, where $\hat{J}_1$ and $\hat J_2$ are
the generalized almost complex structures given by (\ref{n}).
\end{proposition}

\begin{proof}
We have $\hat J_1^2=-I$, $\hat J_2^2=-I$, $\hat J_1\hat J_2=-\hat J_2\hat J_1$,
hence the conclusion.
\end{proof}

\begin{proposition}
Let $(J_1,J_2)$ be an almost complex-product structure on $M$, i.e., $J_1$ is an almost complex structure and $J_2$ is an almost product structure such that $J_1J_2=-J_2J_1$.
Then $(\hat{J}_1,\hat{J}_2,\hat{J}_1\hat{J}_2)$ is a generalized almost para-quaternionic structure on $M$, where $\hat{J}_1$ is
the generalized almost complex structure and $\hat{J}_2$ is the generalized almost product structure given by (\ref{n}).
\end{proposition}

\begin{proof}
We have $\hat J_1^2=-I$, $\hat J_2^2=I$, $\hat J_1\hat J_2=-\hat J_2\hat J_1$,
hence the conclusion.
\end{proof}

\begin{proposition}
Let $J_1$ and $J_2$ be two almost product structures on $M$ such that $J_1J_2=-J_2J_1$.
Then $(\hat J_1\hat{J}_2,\hat{J}_1,\hat{J}_2)$ is a generalized almost para-quaternionic structure on $M$, where $\hat{J}_1$ and $\hat J_2$ are
the generalized almost product structures given by (\ref{n}).
\end{proposition}

\begin{proof}
We have $\hat J_1^2=I$, $\hat J_2^2=I$, $\hat J_1\hat J_2=-\hat J_2\hat J_1$,
hence the conclusion.
\end{proof}

Denote by
\begin{equation} \label{m}
\hat J=:\hat{J}_1\hat{J}_2=\left(
                    \begin{array}{cc}
                      J_1J_2 & 0 \\
                      0 & J_1^*J_2^* \\
                    \end{array}
                  \right).
\end{equation}
We can state the following results.

\begin{proposition}\label{kal}
Let $\nabla$ be an affine connection, let $h$ be a non-degenerate symmetric $(0,2)$-tensor field, let $J_1$ and $J_2$ be two $h$-symmetric $(1,1)$-tensor fields on $M$, and let $\hat J$ be given by (\ref{m}).
Then the following assertions are equivalent:

(i) $\hat \nabla \hat J=0$;

(ii) $(\nabla J_i)\circ J_j +J_i\circ (\nabla J_j)=0, \ i,j=1,2, i\neq j;$

(iii) $\hat \nabla^* \hat J=0$.

In particular, if $\nabla J_1=0$ and $\nabla J_2=0$, then $\hat \nabla \hat J=0$ and $\hat \nabla^* \hat J=0$.
\end{proposition}

\begin{proof}
For any $X,Y\in {\Gamma}^{\infty}(TM)$ and $\eta,\beta \in {\Gamma}^{\infty}(T^*M)$, we have
$$(\hat{\nabla}_{X+\eta}\hat{J})(Y+\beta):=\hat{\nabla}_{X+\eta}(\hat{J}(Y+\beta))-\hat{J}(\hat{\nabla}_{X+\eta}(Y+\beta))=$$
$$=\hat{\nabla}_{X+\eta}\Big(J_1J_2Y+J_1^*J_2^*\beta\Big)-
\hat J\Big(\nabla_XY+h(\nabla_X(h^{-1}(\beta)))\Big)=$$
$$=\nabla_X(J_1J_2Y)+h\big(\nabla_X(h^{-1}(J_1^*J_2^*\beta))\big)-\Big(J_1J_2(\nabla_XY)+J_1^*J_2^*(h(\nabla_X(h^{-1}(\beta))))\Big)=$$
$$=(\nabla_X(J_1J_2))Y+h(\nabla_X(h^{-1}(J_1^*J_2^*\beta)))-J_1^*J_2^*h(\nabla_X(h^{-1}(\beta))).$$

Taking into account that
\[
J_1^*J_2^*h=hJ_2J_1
\]
(and, consequently, $h^{-1}J_1^*J_2^*=J_2J_1h^{-1}$), we get
\[
h(\nabla_X(h^{-1}(J_1^*J_2^*\beta)))-J_1^*J_2^*h(\nabla_X(h^{-1}(\beta)))=
\]
\[
=h(\nabla_X((J_2J_1)(h^{-1}(\beta))))-h((J_2J_1)(\nabla_X(h^{-1}(\beta))))
=h\big((\nabla_X(J_2J_1))(h^{-1}(\beta))\big),
\]
hence
\[
(\hat{\nabla}_{X+\eta}\hat{J})(Y+\beta)=(\nabla_X(J_1J_2))Y+h\big((\nabla_X(J_2J_1))(h^{-1}(\beta))\big).
\]

Also:
\[
(\nabla_X(J_1J_2))Y=(\nabla_XJ_1)(J_2Y)+J_1((\nabla_XJ_2)Y)
\]
and we obtain
\[
(\hat{\nabla}_{X+\eta}\hat{J})(Y+\beta)=(\nabla_XJ_1)(J_2Y)+J_1((\nabla_XJ_2)Y)+
\]
\[
+h\big((\nabla_XJ_2)(J_1(h^{-1}(\beta)))+J_2((\nabla_XJ_1)(h^{-1}(\beta)))\big).
\]

Therefore, $\hat \nabla \hat J=0$ if and only if
\[
\left\{
  \begin{array}{ll}
    (\nabla_XJ_1)(J_2Y)+J_1((\nabla_XJ_2)Y)=0 \\
    (\nabla_XJ_2)(J_1(h^{-1}(\beta)))+J_2((\nabla_XJ_1)(h^{-1}(\beta)))=0
  \end{array}
\right.,
\]
for any $X,Y\in \Gamma^{\infty}(TM)$ and $\beta\in \Gamma^{\infty}(T^*M)$, i.e.,
\pagebreak
\[
\left\{
  \begin{array}{ll}
    (\nabla J_1)\circ J_2+J_1\circ (\nabla J_2)=0 \\
    (\nabla J_2)\circ J_1+J_2\circ (\nabla J_1)=0
  \end{array}
\right.
\]
and we get (i) $\Longleftrightarrow$ (ii).

Now
$$(\hat{\nabla}^*_{X+\eta}\hat{J})(Y+\beta):=\hat{\nabla}^*_{X+\eta}(\hat{J}(Y+\beta))-\hat{J}(\hat{\nabla}^*_{X+\eta}(Y+\beta))=$$
$$=\hat{\nabla}^*_{X+\eta}\Big(J_1J_2Y+J_1^*J_2^*\beta\Big)-\hat J\Big(h^{-1}(\nabla_X(h(Y)))+\nabla_X\beta\Big)=$$
$$=h^{-1}(\nabla_X(h(J_1J_2Y)))+\nabla_X(J_1^*J_2^*\beta)-$$$$-\Big(J_1J_2(h^{-1}(\nabla_X(h(Y))))+J_1^*J_2^*(\nabla_X\beta)\Big).$$

Taking into account that
\[
J_1^*J_2^*h=hJ_2J_1
\]
(and, consequently, $h^{-1}J_1^*J_2^*=J_2J_1h^{-1}$), we get
\[
(\hat{\nabla}^*_{X+\eta}\hat{J})(Y+\beta)=h^{-1}\big((\nabla_X(J_2^*J_1^*))(h(Y))\big)+(\nabla_X(J_1^*J_2^*))\beta.
\]

Also:
\[
(\nabla_X(J_1^*J_2^*))\beta=(\nabla_XJ_1^*)(J_2^*\beta)+J_1^*((\nabla_XJ_2^*)\beta)
\]
and we obtain
\[
(\hat{\nabla}^*_{X+\eta}\hat{J})(Y+\beta)=h^{-1}\big((\nabla_XJ_2^*)(J_1^*(h(Y)))+J_2^*((\nabla_XJ_1^*)(h(Y)))\big)+
\]
\[
+(\nabla_XJ_1^*)(J_2^*\beta)+J_1^*((\nabla_XJ_2^*)\beta).
\]

Therefore, $\hat \nabla^* \hat J=0$ if and only if
\[
\left\{
  \begin{array}{ll}
    (\nabla_XJ_2^*)(J_1^*(h(Y)))+J_2^*((\nabla_XJ_1^*)(h(Y)))=0 \\
    (\nabla_XJ_1^*)(J_2^*\beta)+J_1^*((\nabla_XJ_2^*)\beta)=0
  \end{array}
\right.,
\]
for any $X,Y\in \Gamma^{\infty}(TM)$ and $\beta\in \Gamma^{\infty}(T^*M)$, i.e.,
\[
\left\{
  \begin{array}{ll}
    (\nabla J_2^*)\circ J_1^*+J_2^*\circ (\nabla J_1^*)=0\\
    (\nabla J_1^*)\circ J_2^*+J_1^*\circ (\nabla J_2^*)=0
      \end{array}
\right.
\]
and (ii) $\Longleftrightarrow$ (iii) follows since $((\nabla_X J_i^*)\beta)(Y)=\beta((\nabla_XJ_i)Y)$, for any $X,Y\in \Gamma^{\infty}(TM)$ and $\beta \in \Gamma^{\infty}(T^*M)$, $i=1,2$.
\end{proof}

\begin{corollary}
Let $\nabla$ be an affine connection, let $h$ be a non-degenerate symmetric $(0,2)$-tensor field, let $J_1$ and $J_2$ be two $h$-symmetric $(1,1)$-tensor fields on $M$, and let $\hat J$ be given by (\ref{m}).
Then $\hat{\nabla}^{(\alpha)}\hat{J}=0$, for any $\alpha\in \mathbb{R}$, if and only if
\[
(\nabla J_i)\circ J_j +J_i\circ (\nabla J_j)=0, \ i,j=1,2, i\neq j.
\]

In particular, if $\nabla J_1=0$ and $\nabla J_2=0$, then $\hat \nabla^{(\alpha)} \hat J=0$, for any $\alpha\in \mathbb{R}$.
\end{corollary}

\begin{proof}
Since we have
\[
\hat{\nabla}^{(\alpha)}\hat{J}=\frac{1+\alpha}{2}\hat{\nabla}\hat{J}+
\frac{1-\alpha}{2}\hat{\nabla}^*\hat{J},
\]
we get the conclusion from Proposition \ref{kal}.
\end{proof}

\begin{proposition}
Let $\nabla$ be an affine connection, let $h$ be a non-degenerate symmetric $(0,2)$-tensor field, let $J_1$ and $J_2$ be two $h$-symmetric $(1,1)$-tensor fields on $M$, and let $\hat J$ be given by (\ref{m}).
Then
$\hat J$ is $\nabla$-integrable if and only if
\[\left\{
    \begin{array}{ll}
      N_{J_1J_2}=0 \\
      (\nabla_{J_1J_2X}(J_1^*J_2^*))\beta=J_1^*J_2^*((\nabla_{X}(J_1^*J_2^*))\beta)
    \end{array}
  \right.,
\]
for any $X\in \Gamma^{\infty}(TM)$ and $\beta\in \Gamma^{\infty}(T^*M)$.
\end{proposition}

\begin{proof}
For any $X,Y\in {\Gamma}^{\infty}(TM)$, we have
\[
N_{\hat{J}}^{\nabla}(X,Y)=N_{{J_1J_2}}(X,Y),
\]
\[
N_{\hat{J}}^{\nabla}(X,h(Y))=(\nabla_{J_1J_2X}(J_1^*J_2^*))(h(Y))
-J_1^*J_2^*((\nabla_{X}(J_1^*J_2^*))(h(Y))),
\]
\[
N_{\hat{J}}^{\nabla}(h(X),h(Y))=0.
\]
Then the proof is complete.
\end{proof}

On the other hand, we have the following.
\begin{proposition} Let $J$ be a $h$-symmetric and $\nabla$-parallel $(1,1)$-tensor field on $M$ such that the Nijenhuis tensor of $J$, $N_{J}$, vanishes. Then
the generalized structure \linebreak
$\hat J:=\left(
                                     \begin{array}{cc}
                                       J & 0 \\
                                       0 & -J^* \\
                                     \end{array}
                                   \right)
$ is $\nabla$-integrable. Moreover, $\hat \nabla^* \hat J=0$.
\end{proposition}
\begin{proof} A direct computation gives the following:
$$N^\nabla_{\hat J}(X,Y)=N_J(X,Y), \ \ N^\nabla_{\hat J}(h(X),h(Y))=0,$$$$N^\nabla_{\hat J}(X,h(Y))=-(\nabla_{JX}J^*)(h(Y))-J^*((\nabla_{X}J^*)(h(Y))),$$
for any $X,Y\in \Gamma^\infty(TM)$.
Moreover, a direct computation gives $\hat \nabla^* \hat J=0$.
\end{proof}

\section{Examples of generalized almost complex-product and generalized almost hyperproduct structures}

In order to give examples of all type of generalized triple structures, in this section, we shall describe generalized almost complex-product and generalized almost hyperproduct structures. They are obtained by  two almost complex, respectively two almost product, structures, $J_1$ and $J_2$, which commute.
Concerning them, we have the next two results.

\begin{proposition}
Let $J_1$ and $J_2$ be two commuting almost complex structures on $M$ and let $\hat{J}_1:=\left(
                    \begin{array}{cc}
                      J_1 & 0 \\
                      0 & -J_1^* \\
                    \end{array}
                  \right)
$ and $\hat{J}_2:=\left(
                    \begin{array}{cc}
                      J_2 & 0 \\
                      0 & -J_2^* \\
                    \end{array}
                  \right)
$. Then $(\hat{J}_1,\hat{J}_2,\hat{J}_1\hat{J}_2)$ is a generalized almost complex-product structure on $M$. Moreover, if $\nabla$ is an affine connection on $M$, then  $\hat{J}_1$ and $\hat J_2$ are $\hat \nabla$-parallel if and only if ${J}_1$ and $ J_2$ are $\nabla$-parallel.
\end{proposition}
\begin{proof}
We have $\hat J_1\hat J_2=\hat J_2\hat J_1$ and $\hat J_1^2=-I$, $\hat J_2^2=-I$, $(\hat{J}_1\hat{J}_2)^2=\hat{J}_1^2\hat{J}_2^2=I$,
hence the first conclusion. Moreover, for any $X+\eta, Y+\beta \in \Gamma^\infty(TM \oplus T^*M)$, for $i=1,2$, we get $(\hat \nabla_{X+\eta}\hat J_i)(Y+\beta)=(\nabla_XJ_i)Y-(\nabla_XJ_i^*)\beta$ and the proof is complete.
\end{proof}

\begin{proposition}
Let $J_1$ and $J_2$ be two commuting almost product structures on $M$ and let $\hat{J}_1:=\left(
                    \begin{array}{cc}
                      J_1 & 0 \\
                      0 & -J_1^* \\
                    \end{array}
                  \right)
$ and $\hat{J}_2:=\left(
                    \begin{array}{cc}
                      J_2 & 0 \\
                      0 & -J_2^* \\
                    \end{array}
                  \right)
$. Then $(\hat{J}_1,\hat{J}_2,\hat{J}_1\hat{J}_2)$ is a generalized almost hyperproduct structure on $M$. Moreover, if $\nabla$ is an affine connection on $M$, then $\hat{J}_1$ and $\hat J_2$ are $\hat \nabla$-parallel if and only if ${J}_1$ and $ J_2$ are $\nabla$-parallel.
\end{proposition}
\begin{proof}
We have $\hat J_1\hat J_2=\hat J_2\hat J_1$ and $\hat J_1^2=I$, $\hat J_2^2=I$, $(\hat{J}_1\hat{J}_2)^2=\hat{J}_1^2\hat{J}_2^2=I$,
hence the first conclusion. A direct computation, as in previous proposition, gives the second statement.
\end{proof}

\section{Generalized affine connection parallelizing the \\ structures}

We recall the following \cite{blanan}:

\begin{definition} We call $D:\Gamma ^\infty(TM\oplus T^*M) \times \Gamma ^\infty(TM\oplus T^*M) \rightarrow \Gamma ^\infty(TM\oplus T^*M)$ an \textit{affine connection on} $TM\oplus T^*M$, or a \textit{generalized affine connection on} $M$, if it is $\mathbb R$-bilinear and for any $f\in C^\infty (M)$ and any $\sigma, \tau \in \Gamma ^\infty(TM\oplus T^*M)$, we have
\pagebreak
$$D_{f\sigma}\tau=fD_\sigma \tau,$$
$$D_\sigma(f\tau)=\sigma(f)\tau+fD_\sigma \tau,$$
where $(X+\eta)(f):=X(f)$ for $X+\eta \in \Gamma ^\infty(TM\oplus T^*M)$.
\end{definition}

In this section, first we construct an affine connection on $TM \oplus T^*M$ parallelizing a generalized almost para-quaternionic structure, extending results of \cite{ES} to our setting, then we define the generalized Obata connection parallelizing a generalized almost quaternionic structure.

\subsection{The canonical connection of a generalized para-quaternionic structure}

Let $M$ be an $n$-dimensional smooth manifold, let $\hat J_1$ and $\hat J_2$ be two anticommuting generalized almost product structures and let $(\hat J_1, \hat J_2, \hat J=\hat J_1 \hat J_2)$ be the associated generalized almost para-quaternionic structure. Let $\{X_1,\dots,X_n\}$ be a local frame for $TM$ and let $\{\alpha_1,\dots,\alpha_n\}$ be a local frame for $T^*M$. Then $\{\sigma_{ij}:=X_i+\alpha_j\}_{i,j=1,\dots,n}$ is a local frame for $TM\oplus T^*M$. Let us denote by $V_\pm (\hat J_k)$ the $\pm 1$-eigenbundle of $\hat J_k$ for $k=1,2$. We immediately get the following:
$$\eta_{ij}:=\sigma_{ij}+\hat J_1 \sigma_{ij} \in V_+(\hat J_1),$$
$$\hat J_2\eta_{ij}=\hat J_2(\sigma_{ij}+\hat J_1 \sigma_{ij}) \in V_-(\hat J_1),$$
$$\eta_{ij}+\hat J_2 \eta_{ij}=\sigma_{ij}+\hat J_1 \sigma_{ij} +\hat J_2(\sigma_{ij}+\hat J_1 \sigma_{ij}) \in V_+(\hat J_2),$$
$$\eta_{ij}-\hat J_2 \eta_{ij}=\sigma_{ij}+\hat J_1 \sigma_{ij} -\hat J_2(\sigma_{ij}+\hat J_1 \sigma_{ij}) \in V_-(\hat J_2).$$
Moreover, we denote $V_1:=V_+(\hat J_1)$, $V_2:=V_-(\hat J_1)$ and $V_3:=V_+(\hat J_2)$ and we remark that $$TM \oplus T^*M=V_1 \oplus V_2.$$

We generalize as follows the results about the canonical connection of \cite{ES}.

\begin{lemma}\label{1005} Let $(\hat J_1, \hat J_2, \hat J)$ be a generalized almost para-quaternionic structure and let $D$ be an affine connection on $TM \oplus T^*M$. Then the following conditions are equivalent:

(i) $D_\sigma \tau_l\in \Gamma^\infty(V_l)$, for any $\sigma \in \Gamma ^\infty(TM\oplus T^*M)$ and any $\tau_l \in \Gamma ^\infty (V_l)$, $l=1,2,3$;

(ii) $D\hat J_1=D\hat J_2=0.$
\end{lemma}

\begin{proof} It follows immediately from the properties of $D$, of $V_l$ for $l=1,2,3$ and the following:
$$D_\sigma \eta_{ij}=D_\sigma \sigma_{ij}+\hat J_1(D_\sigma \sigma_{ij})+(D_\sigma \hat J_1)\sigma_{ij},$$
$$D_\sigma (\hat J_2 \eta_{ij})=\hat J_2(D_\sigma \eta_{ij})+(D_\sigma \hat J_2)\eta_{ij},$$
$$D_\sigma (\eta_{ij}+\hat J_2 \eta_{ij})=D_\sigma \eta_{ij}+\hat J_2(D_\sigma \eta_{ij})+(D_\sigma \hat J_2)\eta_{ij}.$$
\end{proof}

\begin{theorem} \label{hjp} Let $M$ be a smooth manifold, let $\nabla$ be an affine connection on $M$ and let $(\hat J_1, \hat J_2, \hat J)$ be a generalized almost para-quaternionic structure. Then there exists a unique affine connection $D$ on $TM\oplus T^*M$ with the following properties:

(i) $D\hat J_1=D\hat J_2=D\hat J=0;$

(ii) $T^D(\sigma,\tau)=0$, for any $\sigma \in \Gamma ^\infty (V_1)$ and any $\tau \in \Gamma ^\infty (V_2),$
where $T^D$ is the torsion tensor of $D$ defined by the bracket of the given connection $\nabla$ as:
$$T^D(\sigma,\tau):=D_\sigma \tau-D_\tau \sigma-[\sigma,\tau]_\nabla.$$

The connection $D$ will be called the \textit{canonical connection} of $(\hat J_1, \hat J_2, \nabla).$
\end{theorem}

\begin{proof} Let us suppose such $D$ exists, let $\hat J_1^+, \hat J_1^-$ be the canonical projections on sections of $V_1$ and $V_2$, respectively defined by:
$$2\hat J_1^+\sigma:=2\sigma^+=\sigma +\hat J_1 \sigma, \, \, 2\hat J_1^-\sigma:=2\sigma^-=\sigma -\hat J_1 \sigma, \, \sigma\in\Gamma ^\infty(TM\oplus T^*M).$$

Then, for any $\sigma \in \Gamma ^\infty (V_1)$ and any $\tau \in \Gamma ^\infty (V_2)$, we have
$$0=T^D(\sigma,\tau)=D_\sigma \tau-D_\tau \sigma-\hat J_1^+[\sigma,\tau]_\nabla-\hat J_1^-[\sigma,\tau]_\nabla.$$

By using Lemma \ref{1005} we get
$$D_\sigma \tau=\hat J_1^-[\sigma,\tau]_\nabla, \, \, D_\tau \sigma=\hat J_1^+[\tau,\sigma]_\nabla.$$

It is easy to verify that these conditions characterize $D$. Indeed, if $\sigma, \tau \in \Gamma ^\infty(TM\oplus T^*M)$, we decompose $\sigma=\sigma^+ +\sigma^-$ and $\tau=\tau^+ +\tau^-$ and we get
$$D_\sigma \tau=D_{\sigma^+}\tau^+ +  D_{\sigma^+}\tau^- + D_{\sigma^-}\tau^+ + D_{\sigma^-}\tau^-.$$

Remark that
$$D_{\sigma^+}\tau^-=\hat J_1^-([\sigma^+,\tau^-]_\nabla)=([\sigma^+,\tau^-]_\nabla)^-,$$
$$D_{\sigma^-}\tau^+=\hat J_1^+([\sigma^-,\tau^+]_\nabla)=([\sigma^-,\tau^+]_\nabla)^+,$$
\pagebreak
$$D_{\sigma^+}\tau^+=D_{\sigma^+}(\hat J_2^2 (\hat J_1^+ \tau))=\hat J_2(D_{\sigma^+}(\hat J_2 (\hat J_1^+ \tau)))=\hat J_2(D_{\sigma^+}(\hat J_2 (\hat J_1^+ \tau)))^-=$$$$=\hat J_2 \hat J_1^- [\sigma^+,\hat J_2 (\hat J_1^+ \tau)]_\nabla,$$
$$D_{\sigma^-}\tau^-=D_{\sigma^-}(\hat J_2^2 (\hat J_1^- \tau))=\hat J_2(D_{\sigma^-}(\hat J_2 (\hat J_1^- \tau)))=\hat J_2(D_{\sigma^-}(\hat J_2 (\hat J_1^- \tau)))^+=$$$$=\hat J_2 \hat J_1^+[\sigma^-,\hat J_2 (\hat J_1^- \tau)]_\nabla.$$

Then, by using the fact that $\hat J_2 \hat J_1^-=\hat J_1^+ \hat J_2$ and $\hat J_2 \hat J_1^+=\hat J_1^- \hat J_2$, we get
$$D_\sigma \tau=\hat J_1^+\Big([\sigma^-,\tau^+]_\nabla+\hat J_2 [\sigma^+,\hat J_2 \tau^+]_\nabla\Big)+
\hat J_1^-\Big([\sigma^+,\tau^-]_\nabla+\hat J_2 [\sigma^-,\hat J_2 \tau^-]_\nabla\Big)=$$
$$= \Big([\sigma^-,\tau^+]_\nabla+\hat J_2 [\sigma^+,\hat J_2 \tau^+]_\nabla\Big)^+ + \Big([\sigma^+,\tau^-]_\nabla+\hat J_2 [\sigma^-,\hat J_2 \tau^-]_\nabla\Big)^-.$$
Hence, conditions (i) and (ii) uniquely define the connection $D$.

Now we prove that such $D$ satisfies the two conditions. We can easily verify that $D$ is an affine connection on $TM \oplus T^*M$ and that  (i) holds. Moreover, remark that for any $\eta \in \Gamma^\infty (TM \oplus T^*M)$, we have
$$(\hat J_1 \eta)^+=\hat J_1 \eta^+=\eta^+, \, (\hat J_1 \eta)^-=\hat J_1 \eta^-=-\eta^-,$$
$$(\hat J_2 \eta)^-=\hat J_2 \eta^+, \, (\hat J_2\eta)^+=\hat J_2 \eta^-.$$

Then, for any $\sigma, \tau \in \Gamma^\infty (TM \oplus T^*M)$, we have
$$D_\sigma (\hat J_1 \tau)=\Big([\sigma^-,(\hat J_1\tau)^+]_\nabla+\hat J_2 [\sigma^+,\hat J_2 (\hat J_1\tau)^+]_\nabla\Big)^+ +
 \Big([\sigma^+,(\hat J_1\tau)^-]_\nabla+\hat J_2 [\sigma^-,\hat J_2 (\hat J_1\tau)^-]_\nabla\Big)^-=$$
$$=\Big([\sigma^-,\tau^+]_\nabla+\hat J_2 [\sigma^+,\hat J_2 \tau^+]_\nabla\Big)^+ + \Big([\sigma^+,-\tau^-]_\nabla+\hat J_2 [\sigma^-,-\hat J_2 \tau^-]_\nabla\Big)^-=$$
$$=\hat J_1\Big(\big([\sigma^-,\tau^+]_\nabla+\hat J_2 [\sigma^+,\hat J_2 \tau^+]_\nabla\big)^+ + \big([\sigma^+,\tau^-]_\nabla+\hat J_2 [\sigma^-,\hat J_2 \tau^-]_\nabla\big)^-\Big)=\hat J_1 (D_\sigma \tau),$$
thus, $D\hat J_1=0$, and a similar computation gives $D\hat J_2=0$, so the proof is complete.
\end{proof}

As an application, we immediately obtain the following.

\begin{proposition} Let $(M,h,\nabla)$ be a quasi-statistical manifold and let $J$ be a $h$-symmetric, $\nabla$-parallel $(1,1)$-tensor field on $M$ such that the Nijenhuis tensor of $J$, $N_J$, vanishes. Let $(\hat J_-,\hat J_+, \hat J)$ be the generalized para-quaternionic structure defined by $\hat J_\mp$ and $\hat J$ given by (\ref{op1}) and (\ref{jj}). Then the canonical connection is the dual connection $\hat \nabla^* $ given by (\ref{22}).
\end{proposition}

\begin{proof} It follows immediately from Corollary \ref{th}.
\end{proof}

Let us consider the Fr\"olicher-Nijenhuis bracket of $\hat J_1$ and $\hat J_2$ defined by $\nabla$, for \linebreak $\sigma, \tau \in \Gamma ^\infty(TM\oplus T^*M)$, as:
$$[\hat J_1,\hat J_2]_\nabla(\sigma,\tau):=[\hat J_1 \sigma,\hat J_2 \tau]_\nabla+[\hat J_2 \sigma,\hat J_1\tau]_\nabla+\hat J_1\hat J_2[\sigma,\tau]_\nabla+$$
$$+\hat J_2\hat J_1[\sigma,\tau]_\nabla-\hat J_1[\hat J_2 \sigma,\tau]_\nabla-\hat J_1[\sigma,\hat J_2 \tau]_\nabla-\hat J_2[\hat J_1 \sigma,\tau]_\nabla-\hat J_2[\sigma,\hat J_1\tau]_\nabla.$$
As $\hat J_1$ and $\hat J_2$ anticommute, we get
$$[\hat J_1,\hat J_2]_\nabla(\sigma,\tau)=[\hat J_1 \sigma,\hat J_2 \tau]_\nabla+[\hat J_2 \sigma ,\hat J_1 \tau]_\nabla-$$
$$-\hat J_1[\hat J_2 \sigma,\tau]_\nabla-\hat J_1[\sigma,\hat J_2 \tau]_\nabla-\hat J_2[\hat J_1 \sigma,\tau]_\nabla-\hat J_2[\sigma,\hat J_1\tau]_\nabla.$$

\begin{lemma}\label{I} Let $M$ be a smooth manifold and let $\nabla$ be an affine connection on $M$. Let $(\hat J_1, \hat J_2, \hat J)$ be a generalized almost para-quaternionic structure and let $D$ be the canonical connection of $(\hat J_1, \hat J_2, \nabla)$ on $TM\oplus T^*M$. Then the following assertions hold:

(i) $[\hat J_1,\hat J_2]_\nabla (\sigma, \tau)=2\hat J_2(T^D(\sigma,\tau))$, for any $\sigma, \tau \in \Gamma^\infty (V_1)$;

(ii) $[\hat J_1,\hat J_2]_\nabla (\sigma, \tau)=-2\hat J_2(T^D(\sigma,\tau))$, for any $\sigma, \tau \in \Gamma (V_2)$;

(iii) $[\hat J_1,\hat J_2]_\nabla (\sigma, \tau)=2\hat J_1^-[\sigma,\hat J_2\tau]_\nabla-2\hat J_1^+[\hat J_2 \sigma, \tau]_\nabla$, for any $\sigma \in \Gamma^\infty (V_1), \tau \in \Gamma^\infty (V_2).$
\end{lemma}

\begin{proof} (i) For $\sigma=\sigma^+, \tau=\tau^+ \in \Gamma^\infty (V_1)$, by using the properties of $D$,  we get
$$[\hat J_1,\hat J_2]_\nabla(\sigma^+,\tau^+)=[\sigma^+,\hat J_2 \tau^+]_\nabla+[\hat J_2 \sigma^+ ,\tau^+]_\nabla-\hat J_1[\hat J_2 \sigma^+,\tau^+]_\nabla-\hat J_1[\sigma^+,\hat J_2 \tau^+]_\nabla-$$
$$-\hat J_2[\sigma^+,\tau^+]_\nabla-\hat J_2[\sigma^+,\tau^+]_\nabla=-T^D(\sigma^+,\hat J_2 \tau^+)-T^D(\hat J_2 \sigma^+ ,\tau^+)+$$$$+\hat J_1 T^D(\hat J_2 \sigma^+,\tau^+)+\hat J_1 T^D(\sigma^+,\hat J_2\tau^+)+\hat J_2T^D(\sigma^+,\tau^+)+\hat J_2T^D(\sigma^+,\tau^+)=2\hat J_2 T^D(\sigma,\tau).$$

(ii) For $\sigma=\sigma^-, \tau=\tau^- \in \Gamma^\infty (V_2)$, by using the properties of $D$,  we get
$$[\hat J_1,\hat J_2]_\nabla(\sigma^-,\tau^-)=[-\sigma^-,\hat J_2 \tau^-]_\nabla+[\hat J_2 \sigma^- ,-\tau^-]_\nabla-\hat J_1[\hat J_2 \sigma^-,\tau^-]_\nabla-\hat J_1[\sigma^-,\hat J_2 \tau^-]_\nabla-$$
$$-\hat J_2[-\sigma^-,\tau^-]_\nabla-\hat J_2[\sigma^-,-\tau^-]_\nabla=T^D(\sigma^-,\hat J_2 \tau^-)+T^D(\hat J_2 \sigma^- ,\tau^-)+$$$$+\hat J_1 T^D(\hat J_2 \sigma^-,\tau^-)+\hat J_1 T^D(\sigma^-,\hat J_2\tau^-)-\hat J_2T^D(\sigma^-,\tau^-)-\hat J_2T^D(\sigma^-,\tau^-)=-2\hat J_2 T^D(\sigma,\tau).$$

(iii) Finally, for $\sigma=\sigma^+\in \Gamma^\infty (V_1), \tau=\tau^- \in \Gamma^\infty (V_2)$, by using the properties of $D$,  we get
$$[\hat J_1,\hat J_2]_\nabla(\sigma^+,\tau^-)=[\sigma^+,\hat J_2 \tau^-]_\nabla+[\hat J_2 \sigma^+ ,-\tau^-]_\nabla-\hat J_1[\hat J_2 \sigma^+,\tau^-]_\nabla-\hat J_1[\sigma^+,\hat J_2 \tau^-]_\nabla-$$
$$-\hat J_2[\sigma^+,\tau^-]_\nabla-\hat J_2[\sigma^+,-\tau^-]_\nabla=-T^D(\sigma^+,\hat J_2 \tau^-)+T^D(\hat J_2 \sigma^+ ,\tau^-)+$$$$+\hat J_1 T^D(\hat J_2 \sigma^+,\tau^-)+\hat J_1 T^D(\sigma^+,\hat J_2\tau^-)+\hat J_2T^D(\sigma^+,\tau^-)-\hat J_2T^D(\sigma^+,\tau^-)=$$
\pagebreak
$$=\hat J_1 T^D(\sigma^+,\hat J_2\tau^-)-\hat J_1T^D(\sigma^+,\hat J_2\tau^-)+\hat J_1 T^D(\hat J_2 \sigma^+,\tau^-)+\hat J_1T^D(\hat J_2 \sigma^+,\tau^-)=$$
$$=2\hat J_1^-[\sigma,\hat J_2\tau]_\nabla-2\hat J_1^+[\hat J_2 \sigma, \tau]_\nabla$$
and the proof is complete.
\end{proof}

\begin{theorem} Let $M$ be a smooth manifold and let $\nabla$ be an affine connection on $M$. Let $(\hat J_1, \hat J_2, \hat J)$ be a generalized almost para-quaternionic structure and let $D$ be the canonical connection of $(\hat J_1, \hat J_2, \nabla)$ on $TM\oplus T^*M$. Then the following conditions are equivalent:

(i) $[\hat J_1,\hat J_2]_\nabla =0$;

(ii) $D$ is torsion-free;

(iii) $\hat J_1$ and $\hat J_2$ are $\nabla$-integrable.

\end{theorem}

\begin{proof} Let us suppose that  $[\hat J_1,\hat J_2]_\nabla =0$. Then, from Lemma \ref{I}, we get $T^D(\sigma,\tau)=0$ for any $\sigma, \tau \in  \Gamma^\infty (V_1)$ and  for any $\sigma, \tau \in \Gamma^\infty (V_2)$. On the other hand, for $\sigma \in  \Gamma^\infty(V_1)$ and $\tau \in \Gamma^\infty(V_2)$, $T^D(\sigma,\tau)$ vanishes, as $D$ is the canonical connection. Then $T^D=0$.

Now let us suppose that $T^D=0$. Then, for any $\sigma, \tau \in\Gamma^\infty (TM\oplus T^*M)$, we get
$[\sigma,\tau]_\nabla=D_{\sigma} \tau -D_{\tau} \sigma$ and by using the fact that $D\hat J_1=D\hat J_2=0$, we immediately get $N^\nabla_{\hat J_1}=0$ and $N^\nabla_{\hat J_2}=0$. Also, remark that from Proposition \ref{ii} we get that $\hat J$ is $\nabla$-integrable, too.

Now, let us suppose that $\hat J_1$ and $\hat J_2$ are $\nabla$-integrable. Then, for $\sigma=\sigma^+, \tau=\tau^+ \in \Gamma^\infty (V_1)$ we get
$$N^\nabla_{\hat J_1}(\sigma^+,\tau^+)=2\hat J_1T^D(\sigma^+,\tau^+)=0$$
and from Lemma \ref{I}, $[\hat J_1,\hat J_2](\sigma,\tau)=0$ for any $\sigma,\tau \in \Gamma^\infty (V_1)$.

Moreover, for $\sigma=\sigma^-, \tau=\tau^- \in \Gamma^\infty (V_2)$, we get
$$N^\nabla_{\hat J_2}(\sigma^-,\tau^-)=-T^D(\hat J_2 \sigma^-,\hat J_2 \tau^-)-T^D(\sigma^-,\tau^-)=0$$
and from Lemma \ref{I}, $[\hat J_1,\hat J_2](\sigma,\tau)=0$ for any $\sigma,\tau \in \Gamma^\infty (V_2)$.

Finally, for any $\sigma =\sigma^+ \in \Gamma ^\infty(V_1),\tau=\tau^- \in \Gamma^\infty (V_2)$, we have
$$[\hat J_1,\hat J_2]_\nabla (\sigma, \tau)=2\hat J_1^-[\sigma^+,\hat J_2\tau^-]_\nabla-2\hat J_1^+[\hat J_2 \sigma^+, \tau^-]_\nabla-=$$
$$=2\hat J_1^-\Big(D_{\sigma^+}\hat J_2 \tau^- -D_{\hat J_2 \tau^-} \sigma^+-T^D(\sigma^+,\hat J_2 \tau^-)\Big)-$$
$$-2\hat J^+_1\Big(D_{\hat J_2 \sigma^+}\tau^- - D_{\tau^-}\hat J_2\sigma^+-T^D(\hat J_2\sigma^+,\tau^-)\Big)=$$
$$=2\hat J_1^-\hat J_2D_{\sigma^+} \tau^- -2\hat J_1^-D_{\hat J_2\tau^-}   \sigma^+ -2\hat J^+_1D_{\hat J_2 \sigma^+}\tau^- +2\hat J^+_1\hat J_2D_{\tau^-}\sigma^+.$$

Now, remark that $D_{\sigma^+} \tau^- \in \Gamma^\infty(V_2)$, then $\hat J_2D_{\sigma^+} \tau^-\in \Gamma^\infty (V_1)$, hence $\hat J_1^-\hat J_2D_{\sigma^+} \tau^- =0$.

Analogously, $D_{\hat J_2\tau^-}   \sigma^+\in \Gamma^\infty (V_1)$, hence $\hat J_1^-D_{\hat J_2\tau^-}   \sigma^+ =0$,
$D_{\hat J_2 \sigma^+}\tau^-$, $\hat J_2D_{\tau^-}\sigma^+\in \Gamma^\infty (V_2)$, hence $\hat J^+_1D_{\hat J_2 \sigma^+}\tau^-=0$, $\hat J^+_1\hat J_2D_{\tau^-}\sigma^+=0$.

Thus $[\hat J_1,\hat J_2]_\nabla (\sigma, \tau)=0$ for any $\sigma, \tau \in \Gamma^\infty(TM \oplus T^*M)$, and the proof is complete.
\end{proof}

\subsection{The canonical connection of a generalized quaternionic structure}

Due to Obata's theorem \cite{ob}, we know that, for an integrable almost quaternionic structure, there exists a unique torsion-free affine connection that makes it parallel, and called the \textit{canonical connection}.

For a generalized almost quaternionic structures $(\hat J_1,\hat J_2,\hat J_3)$ on a smooth manifold $M$, by using the explicit expression of the Obata connection given in \cite{AM}, we define an affine connection on $TM \oplus T^*M$, $D$, that we call the \textit{generalized Obata connection}, as follows:
\begin{equation}  \label{1000}
D_\sigma \tau ={1\over 12}\big\{ \sum_{\alpha, \beta, \gamma}\hat J_\alpha([\hat J_\beta \sigma, \hat J_\gamma \tau]_\nabla+[\hat J_\beta \tau, \hat J_\gamma \sigma]_\nabla)+2\sum_{\alpha =1} ^3 \hat J_\alpha([\hat J_\alpha \sigma, \tau]_\nabla+[\hat J_\alpha \tau, \sigma]_\nabla)-\end{equation}
$$-\sum_{\alpha =1} ^3 N^\nabla_{\hat J_\alpha}(\sigma,\tau)\big \}+ {1\over 2}[\sigma, \tau]_\nabla,$$
where $\nabla$ is a given affine connection on $M$, $\sigma,\tau\in \Gamma^\infty(TM\oplus T^*M)$ and $\alpha, \beta, \gamma$ in the first sum run in cyclic permutations of $1,2,3$.

\begin{proposition}\label{DDD} Let $M$ be a smooth manifold, let $\nabla$ be an affine connection on $M$, let $(\hat J_1,\hat J_2,\hat J_3)$ be a generalized almost quaternionic structure and let $D$ be the generalized Obata connection. Then:
$$D\hat J_1=D\hat J_2=D\hat J_3=0.$$
\end{proposition}

\begin{proof} For any $\sigma,\tau \in \Gamma^\infty (TM \oplus T^*M)$ the generalized Obata connection (\ref{1000}) is given by:
$$D_\sigma \tau={1\over 12}\big\{ \hat J_1[\hat J_2 \sigma, \hat J_3\tau]_\nabla-\hat J_1[\hat J_3 \sigma, \hat J_2 \tau]_\nabla+ \hat J_2[\hat J_3 \sigma, \hat J_1\tau]_\nabla-\hat J_2[\hat J_1 \sigma, \hat J_3 \tau]_\nabla+ \hat J_3[\hat J_1 \sigma, \hat J_2\tau]_\nabla-$$
$$-\hat J_3[\hat J_2 \sigma, \hat J_1 \tau]_\nabla+ \hat J_1[\hat J_1 \sigma, \tau]_\nabla-3\hat J_1[\sigma, \hat J_1 \tau]_\nabla+ \hat J_2[\hat J_2 \sigma, \tau]_\nabla-3\hat J_2[\sigma, \hat J_2 \tau]_\nabla+ \hat J_3[\hat J_3 \sigma, \tau]_\nabla-$$
$$-3\hat J_3[\sigma, \hat J_3 \tau]_\nabla+
[\hat J_1 \sigma, \hat J_1\tau]_\nabla+[\hat J_2 \sigma, \hat J_2 \tau]_\nabla+[\hat J_3 \sigma, \hat J_3 \tau]_\nabla+3[\sigma,\tau]_\nabla \big \}.$$

Direct computations give $D_\sigma (\hat J_1 \tau)=\hat J_1(D_\sigma \tau)$, $D_\sigma(\hat J_2 \tau)=\hat J_2(D_\sigma \tau)$, hence $D\hat J_1=D\hat J_2=0$ and furthermore $D(\hat J_1 \hat J_2)=D\hat J_3=0$, so the proof is complete.
\end{proof}

\begin{proposition}\label{1006} The torsion $T^D$ of the generalized Obata connection is given by:
$$T^D(\sigma,\tau)={1\over 6} \sum_{\alpha =1}^3 N^\nabla_{\hat J_\alpha}(\sigma,\tau),$$
for any $\sigma,\tau \in \Gamma^\infty (TM \oplus T^*M)$.
\end{proposition}

\begin{proof} From the definition of the torsion tensor of $D$ defined by the bracket of the given connection $\nabla$ as $T^D(\sigma,\tau):=D_\sigma \tau -D_\tau \sigma -[\sigma,\tau]_\nabla$, a direct computation gives the statement.
\end{proof}

\begin{corollary}  Let $M$ be a smooth manifold, let $\nabla$ be an affine connection on $M$ and let $(\hat J_1,\hat J_2,\hat J_3)$ be a generalized almost quaternionic structure. Then $\hat J_1,\hat J_2,\hat J_3$ are $\nabla$-integrable if and only if the generalized Obata connection is torsion-free.
\end{corollary}

\begin{proof} Indeed, if  $\hat J_1,\hat J_2,\hat J_3$ are $\nabla$-integrable then, from Proposition \ref{1006}, we get that $T^D=0$. On the other hand, as $D$ parallelizes  $\hat J_1$, $\hat J_2$ and $\hat J_3$, the vanishing of the torsion tensor implies that the Nijenhuis tensor of $\hat J_i$ is $0$ for $i=1,2,3$.
\end{proof}

Let us denote $E=TM \oplus T^*M$ and $E^{\mathbb C}=(TM \oplus T^*M)\otimes _{\mathbb R}\mathbb C$. The splitting in $\pm \sqrt{-1}$ eigenspaces of $\hat J_i$, $i=1,2,3$, is $E^{\mathbb C}=E^{1,0}_i \oplus E^{0,1}_i$, where
$$E^{1,0}_i =\{\sigma -\sqrt {-1} \hat J_i \sigma \,  \vert \, \sigma \in E\}, \, \, E^{0,1}_i =\overline{E^{1,0}_i}=\{\sigma +\sqrt {-1} \hat J_i \sigma \, \vert \, \sigma \in E\}.$$

Let $V_1=E^{1,0}_1$, $V_2=E^{0,1}_1$ and $V_3=E^{1,0}_2$.
We have immediately that $\hat J_2\tau\in \Gamma^\infty (V_2)$, for any $\tau \in \Gamma^\infty (V_1)$. Finally, let us define the canonical projections:
$$\hat J_i^-:E \rightarrow E^{1,0}_i ,\, \, \hat J_i^+:E \rightarrow E^{0,1}_i, \, \, \hat J_i^\mp \tau:=\tau_i^\mp={1\over 2}(\tau\mp \sqrt {-1} \hat J_i \tau).$$

A direct computation gives immediately the following.
\begin{lemma}\label{1007} For any $\sigma\in \Gamma ^\infty (E)$ we have
\[
(\hat J_1 \sigma)^\mp=\hat J_1 \sigma ^\mp=\pm\sqrt {-1} \sigma ^\mp, \ \ (\hat J_2 \sigma)^\mp=\hat J_2 \sigma ^\pm, \ \ (\hat J_3 \sigma)^\mp=\mp\sqrt {-1} (\hat J_2\sigma ^\mp).
\]
\end{lemma}

\begin{lemma}\label{1010} Let $(\hat J_1, \hat J_2, \hat J)$ be a generalized almost quaternionic structure, let $D$ be an affine connection on $E=TM \oplus T^*M$ and consider the extension of $D$ on $E^{\mathbb C}$. Then the following conditions are equivalent:

(i) $D_\sigma \tau_i \in \Gamma^\infty(V_i)$, for any $\sigma \in \Gamma ^\infty(TM\oplus T^*M)$ and any $\tau_i \in \Gamma ^\infty (V_i)$, $i=1,2,3$;

(ii) $D\hat J_1=D\hat J_2=0.$
\end{lemma}

\begin{proof} It follows immediately from the properties of $D$, of $V_i$ for $i=1,2,3$ and the following:
$$D_\sigma (\tau \mp \sqrt{-1}\hat J_1 \tau)=D_\sigma \tau \mp \sqrt {-1}\hat J_1(D_\sigma \tau)\mp(D_\sigma \hat J_1)\tau,$$
$$D_\sigma (\hat J_2 (\tau \mp \sqrt{-1}\hat J_1 \tau))=\hat J_2(D_\sigma(\tau \mp \sqrt{-1}\hat J_1 \tau))+(D_\sigma \hat J_2)(\tau \mp \sqrt{-1}\hat J_1 \tau),$$
$$D_\sigma (\tau \mp \sqrt{-1}\hat J_2 \tau)=D_\sigma \tau \mp \sqrt {-1}\hat J_2(D_\sigma \tau)\mp(D_\sigma \hat J_2)\tau,$$
where $\tau \in \Gamma^\infty(TM \oplus T^*M)$.
\end{proof}

\begin{theorem} \label{HJP} Let $M$ be a smooth manifold, let $\nabla$ be an affine connection on $M$ and let $(\hat J_1, \hat J_2, \hat J)$ be a generalized quaternionic structure. Then there exists a unique affine connection $D$ on $TM\oplus T^*M$ with the following properties:

(i) $D\hat J_1=D\hat J_2=D\hat J=0;$

(ii) $T^D(\sigma,\tau)=0$, for any $\sigma \in \Gamma ^\infty (V_1)$ and any $\tau \in \Gamma ^\infty (V_2),$
where $T^D$ is the torsion tensor of the extension of $D$ to $E^{\mathbb C}$ defined by the bracket of the given connection $\nabla$ as:
$$T^D(\sigma,\tau):=D_\sigma \tau-D_\tau \sigma-[\sigma,\tau]_\nabla.$$

The connection $D$ will be called the \textit{canonical connection} of $(\hat J_1, \hat J_2, \nabla).$
\end{theorem}

\begin{proof} We proceed essentially as in the para-quaternionic case, with the appropriate changes. Let us suppose such $D$ exists, let $\hat J_1^-, \hat J_1^+$ be the canonical projections on sections of $V_1$ and $V_2$.

Then, for any $\sigma \in \Gamma ^\infty (V_1)$ and any $\tau \in \Gamma ^\infty (V_2)$, we have
$$0=T^D(\sigma,\tau)=D_\sigma \tau-D_\tau \sigma-\hat J_1^-[\sigma,\tau]_\nabla-\hat J_1^+[\sigma,\tau]_\nabla.$$

By using Lemma \ref{1010} we get
$$D_\sigma \tau=\hat J_1^+[\sigma,\tau]_\nabla, \, \, D_\tau \sigma=\hat J_1^-[\tau,\sigma]_\nabla.$$

These conditions characterize $D$. Indeed, if $\sigma, \tau \in \Gamma ^\infty(TM\oplus T^*M)$, we decompose $\sigma=\sigma^+ +\sigma^-$ and $\tau=\tau^+ +\tau^-$ and we get
$$D_\sigma \tau=D_{\sigma^-}\tau^- +  D_{\sigma^-}\tau^+ + D_{\sigma^+}\tau^- + D_{\sigma^+}\tau^+.$$

Remark that
$$D_{\sigma^+}\tau^-=\hat J_1^-([\sigma^+,\tau^-]_\nabla)=([\sigma^+,\tau^-]_\nabla)^-,$$
$$D_{\sigma^-}\tau^+=\hat J_1^+([\sigma^-,\tau^+]_\nabla)=([\sigma^-,\tau^+]_\nabla)^+,$$
$$D_{\sigma^+}\tau^+=D_{\sigma^+}(-\hat J_2^2  \tau^+)=-\hat J_2(D_{\sigma^+}(\hat J_2 \tau^+))=-\hat J_2(D_{\sigma^+}(\hat J_2 \tau^+)^-)=$$$$=-\hat J_2([\sigma^+,\hat J_2 \tau^+)]_\nabla)^-,$$
$$D_{\sigma^-}\tau^-=D_{\sigma^-}(-\hat J_2^2 \tau^-)=-\hat J_2(D_{\sigma^-}(\hat J_2  \tau ^-))=-\hat J_2(D_{\sigma^-}(\hat J_2 \tau ^-)^+)=$$$$=-\hat J_2 ([\sigma^-,\hat J_2  \tau ^-]_\nabla)^+.$$

Then, by using the fact that $\hat J_2 \hat J_1^-=\hat J_1^+ \hat J_2$ and $\hat J_2 \hat J_1^+=\hat J_1^- \hat J_2$, we get
$$D_\sigma \tau=\hat J_1^+\Big([\sigma^-,\tau^+]_\nabla-\hat J_2 [\sigma^+,\hat J_2 \tau^+]_\nabla\Big)+
\hat J_1^-\Big([\sigma^+,\tau^-]_\nabla-\hat J_2 [\sigma^-,\hat J_2 \tau^-]_\nabla\Big)=$$
$$= \Big([\sigma^-,\tau^+]_\nabla-\hat J_2 [\sigma^+,\hat J_2 \tau^+]_\nabla\Big)^+ + \Big([\sigma^+,\tau^-]_\nabla-\hat J_2 [\sigma^-,\hat J_2 \tau^-]_\nabla\Big)^-.$$
Hence, conditions (i) and (ii) uniquely define the connection $D$.

Now we prove that such $D$ satisfies the two conditions. We can easily verify that $D$ is an affine connection on $TM \oplus T^*M$ and that  (i) holds. Moreover, by using Lemma \ref{1007}, for any  $\sigma, \tau \in \Gamma^\infty (TM \oplus T^*M)$, we have
$$D_\sigma (\hat J_1 \tau)=\Big([\sigma^-,(\hat J_1\tau)^+]_\nabla-\hat J_2 [\sigma^+,\hat J_2 (\hat J_1\tau)^+]_\nabla\Big)^+ +
 \Big([\sigma^+,(\hat J_1\tau)^-]_\nabla-\hat J_2 [\sigma^-,\hat J_2 (\hat J_1\tau)^-]_\nabla\Big)^-=$$
$$=-\sqrt{-1}\Big([\sigma^-,\tau^+]_\nabla+\hat J_2 [\sigma^+,\hat J_2 \tau^+]_\nabla\Big)^+ +\sqrt{-1} \Big([\sigma^+,-\tau^-]_\nabla+\hat J_2 [\sigma^-,-\hat J_2 \tau^-]_\nabla\Big)^-=$$
$$=\hat J_1\Big(\big([\sigma^-,\tau^+]_\nabla+\hat J_2 [\sigma^+,\hat J_2 \tau^+]_\nabla\big)^+ + \big([\sigma^+,\tau^-]_\nabla+\hat J_2 [\sigma^-,\hat J_2 \tau^-]_\nabla\big)^-\Big)=\hat J_1 (D_\sigma \tau),$$
thus, $D\hat J_1=0$. With a similar computation we get
$$D_\sigma (\hat J_2 \tau)=\Big([\sigma^-,(\hat J_2\tau)^+]_\nabla-\hat J_2 [\sigma^+,\hat J_2 (\hat J_2\tau)^+]_\nabla\Big)^+ +
 \Big([\sigma^+,(\hat J_2\tau)^-]_\nabla-\hat J_2 [\sigma^-,\hat J_2 (\hat J_2\tau)^-]_\nabla\Big)^-=$$
$$=\Big([\sigma^-,\hat J_2\tau^-]_\nabla+\hat J_2 [\sigma^+,\tau^-]_\nabla\Big)^+ +\Big([\sigma^+,\hat J_2\tau^+]_\nabla+\hat J_2 [\sigma^-,\tau^+]_\nabla\Big)^-=$$
$$=\hat J_2\Big(\Big(-\hat J_2[\sigma^-,\hat J_2\tau^-]_\nabla+ [\sigma^+,\tau^-]_\nabla\Big)^-+\Big(-\hat J_2[\sigma^+,\hat J_2\tau^+]_\nabla+ [\sigma^-,\tau^+]_\nabla\Big)^+\Big)=\hat J_2 (D_\sigma \tau),$$
thus, $D\hat J_2=0$.
Then the proof is complete.
\end{proof}

As an application we get the following.

\begin{proposition} Let $M$ be a smooth manifold, let $\nabla$ be an affine connection on $M$ and let $(\hat J_1, \hat J_2, \hat J)$ be a generalized quaternionic structure. Then the canonical connection is the generalized Obata connection.
\end{proposition}

\begin{proof} If $\hat J_1, \hat J_2, \hat J$ are $\nabla$-integrable then, from Proposition \ref{1006}, the generalized Obata connection is torsion-free and, from Proposition \ref{DDD}, it parallelizes the quaternionic structure, hence it satisfies the hypothesis of Theorem \ref{HJP}, in particular it is the canonical connection.
\end{proof}

\subsection{Examples}

\begin{proposition}
Let $(h,J_1)$ and $(h,J_2)$ be two Norden (or, a Norden and a para-Norden, or, two para-Norden) structures such that $J_1J_2=-J_2J_1$, and let $\nabla$ be the
canonical connection of $(J_1,J_2,J_1J_2)$. If $(h,\nabla)$ is a statistical structure, then $\hat J_1$ and $\hat J_2$ are $\nabla$-integrable,
i.e., $(\hat J_1,\hat J_2, \hat J_1\hat J_2)$ is a generalized quaternionic (respectively, generalized para-quaternionic) structure, where $\hat J_1$ and $\hat J_2$ are given either by (\ref{on}) or by (\ref{n}).
\end{proposition}

\begin{proof}
It follows from the fact that $N_{J_1}=N_{J_2}=0$ implies $N_{\hat J_1}^{\nabla}=N_{\hat J_2}^{\nabla}=0$.
\end{proof}

Based on this result, we shall prove the following.
\begin{proposition}
Let $(h,J_1)$ and $(h,J_2)$ be two Norden (respectively, a Norden and a para-Norden, respectively, two para-Norden) structures on $M$ such that $J_1J_2=-J_2J_1$, and let $\nabla$ be the canonical connection of $(J_1,J_2,J_1J_2)$. If $(h,\nabla)$ is a statistical structure, then
the canonical connection of $(\hat J_1,\hat J_2, \hat J_1\hat J_2)$ is the dual connection $\hat \nabla^*$ given by (\ref{22}),
%is the unique torsion-free generalized affine connection such that $\hat \nabla^* \hat J_1=\hat \nabla^* \hat J_2=0$,
where $\hat J_1$ and $\hat J_2$ are given either by (\ref{on}) or by (\ref{n}).
\end{proposition}

\begin{proof}
By a direct computation, in both of the two cases, we get for any $X,Y\in \Gamma(TM)$ and $\eta,\beta\in \Gamma(T^*M)$
\[
(\hat \nabla^*_{X+\eta} \hat J_i)(Y+\beta):=\hat \nabla^*_{X+\eta}(\hat J_i(Y+\beta))-\hat J_i(\hat \nabla^*_{X+\eta}Y+\beta)=
\]
\[
=h^{-1}((\nabla_XJ_i^*)(h(Y)))-(\nabla_XJ_i^*)\beta, \ i=1,2.
\]

Now taking into account that
\[
\big((\nabla_XJ_i^*)\eta\big)(Y):=(\nabla_X(J_i^*\eta))(Y)-J_i^*(\nabla_X\eta)(Y):=
\]
\[
:=X((J_i^*\eta)(Y))-(J_i^*\eta)(\nabla_XY)-(\nabla_X\eta)(JY)
=X(\eta(JY))-\eta(J(\nabla_XY))-(\nabla_X\eta)(JY)=
\]
\[
=X(\eta(JY))+\eta((\nabla_XJ)Y)-\eta(\nabla_XJY)-(\nabla_X\eta)(JY)
=\eta\big((\nabla_XJ_i)Y\big),
\]
for any $X,Y\in \Gamma(TM)$ and $\eta\in \Gamma(T^*M)$, we get the conclusion from Theorems \ref{hjp} and Theorems \ref{HJP}.
\end{proof}

\subsection{Families of triple structures}

Finally we show that the generalized affine connection that parallelizes certain families of generalized complex and product structures is preserved.

\begin{proposition} \label{kjh}
Let $(\hat J, \hat K,\hat J\hat K)$ be a generalized almost para-quaternionic structure and let $a,b\in \mathbb R$. Then

(i) $\hat J_{a,b}:=a\hat J+b\hat J\hat K$ is a generalized almost complex structure if $a^2-b^2=1$ and a generalized almost product structure if $b^2-a^2=1$;

(ii) $\hat K_{a,b}:=a\hat K+b\hat J\hat K$ is a generalized almost product structure if $a^2+b^2=1$.
\end{proposition}

\begin{proof}
By direct computations and taking into account that $\hat J\hat K=-\hat K\hat J$, we consequently get
\[
\hat J_{a,b}^2=-(a^2-b^2)I, \ \ \hat K_{a,b}^2=(a^2+b^2)I,
\]
hence the conclusions.
\end{proof}

As consequences from Proposition \ref{kjh}, we obtain

\begin{corollary}
If $(\hat J,\hat K,\hat J\hat K)$ is a generalized almost para-quaternionic structure, then

(i) $\hat J^{\pm}_{\theta}:=\cosh \theta \cdot \hat J\pm \sinh \theta \cdot \hat J\hat K$ is a generalized almost complex structure, for any $\theta \in \mathbb R$;

(ii) $\hat K^{\pm}_{1, \theta}:=\sinh \theta \cdot \hat J\pm \cosh \theta \cdot \hat J\hat K$, \ \ \ $\hat K^{\pm}_{2, \theta}:=\cos \theta \cdot \hat J\pm \sin \theta \cdot \hat J\hat K$ \ \ \ and \linebreak $\hat K^{\pm}_{3, \theta}:=\sin \theta \cdot \hat J\pm \cos \theta \cdot \hat J\hat K$ are generalized almost product structures, for any $\theta \in \mathbb R$.
\end{corollary}

Concerning the associated connections, we prove the following result.

\begin{proposition}
Let $(\hat J, \hat K,\hat J\hat K)$ be a generalized para-quaternionic structure on $M$, let $(\hat J_{a,b},\hat K)_{a,b\in \mathbb R, \ a^2-b^2=\pm 1}$ and let $(\hat J, \hat K_{a,b})_{a,b\in \mathbb R, \ a^2+b^2=1}$ as above.
Then the canonical connections of $(\hat J,\hat K)$, $(\hat J_{a,b},\hat K)$ and $(\hat J, \hat K_{a,b})$ coincide.
\end{proposition}

\begin{proof}
Let $D $ be the canonical connection of $(\hat J,\hat K,\hat J\hat K)$.
Then, for any $X,Y\in \Gamma(TM)$ and $\eta,\beta\in \Gamma(T^*M)$, we have
\[
(D _{X+\eta}\hat J_{a,b})(Y+\beta):=D _{X+\eta}(\hat J_{a,b}(Y+\beta))-\hat J_{a,b}(D _{X+\eta}Y+\beta)=
\]
\[
=a(D _{X+\eta}\hat J)(Y+\beta)+b\Big((D _{X+\eta}\hat J)(\hat K(Y+\beta))+\hat J((D _{X+\eta}\hat K)(Y+\beta))\Big)
\]
and
\[
(D _{X+\eta}\hat K_{a,b})(Y+\beta):=D _{X+\eta}(\hat K_{a,b}(Y+\beta))-\hat K_{a,b}(D _{X+\eta}Y+\beta)=
\]
\[
=a(D _{X+\eta}\hat K)(Y+\beta)+b\Big((D _{X+\eta}\hat J)(\hat K(Y+\beta))+\hat J((D _{X+\eta}\hat K)(Y+\beta))\Big),
\]
and the conclusion follows from Theorem \ref{hjp}.
\end{proof}

Similarly, we will obtain the following.

\begin{proposition} \label{kjh1}
(i) Let $(\hat J_1, \hat J_2,\hat J_1\hat J_2)$ be a generalized almost quaternionic structure and let $a,b\in \mathbb R$. Then $\hat J_{a,b}:=a\hat J_1+b\hat J_1\hat J_2$ is a generalized almost complex structure if $a^2+b^2=1$.

(ii) Let $(\hat K_1, \hat K_2,\hat K_1\hat K_2)$ be a generalized almost para-quaternionic structure and let $a,b\in \mathbb R$. Then $\hat K_{a,b}:=a\hat K_1+b\hat K_1\hat K_2$ is a generalized almost complex structure if $b^2-a^2=1$ and a generalized almost product structure if $a^2-b^2=1$.
\end{proposition}

\begin{proof}
By direct computations, and taking into account that $\hat J_1\hat J_2=-\hat J_2\hat J_1$ and $\hat K_1\hat K_2=-\hat K_2\hat K_1$, we consequently get
\[
\hat J_{a,b}^2=-(a^2+b^2)I, \ \ \hat K_{a,b}^2=(a^2-b^2)I,
\]
hence the conclusions.
\end{proof}

As consequences from Proposition \ref{kjh1}, we obtain the following two corollaries.

\begin{corollary}
Let $\hat J_1$ and $\hat J_2$ be two anticommuting generalized almost complex structures. Then
$\hat J^{\pm}_{1,\theta}:=\cos \theta \cdot \hat J_1\pm \sin \theta \cdot \hat J_1\hat J_2$
and $\hat J^{\pm}_{2,\theta}:=\sin \theta \cdot \hat J_1\pm \cos \theta \cdot \hat J_1\hat J_2$ are generalized almost complex structures, for any $\theta \in \mathbb R$.
\end{corollary}

\begin{corollary}
Let $\hat K_1$ and $\hat K_2$ be two anticommuting generalized almost product structures. Then

(i) $\hat K^{\pm}_{\theta}:=\sinh \theta \cdot \hat K_1\pm \cosh \theta \cdot \hat K_1\hat K_2$ is a generalized almost complex structure, for any $\theta \in \mathbb R$;

(ii) $\hat K^{\pm}_{\theta}:=\cosh \theta \cdot \hat K_1\pm \sinh \theta \cdot \hat K_1\hat K_2$ is a generalized almost product structure, for any $\theta \in \mathbb R$.
\end{corollary}

Concerning the associated connections, we prove the following results.

\begin{proposition}
Let $(\hat J_1, \hat J_2,\hat J_1\hat J_2)$ be a generalized quaternionic structure on $M$, let $(\hat J_{a,b},\hat J_1)_{a,b\in \mathbb R, \ a^2+b^2=1}$ and let $(\hat J_{a,b}, \hat J_2)_{a,b\in \mathbb R, \ a^2+b^2=1}$ as above.
Then the canonical connections of $(\hat J_1,\hat J_2)$, $(\hat J_{a,b},\hat J_1)$ and $(\hat J_{a,b},\hat J_2)$ coincide.
\end{proposition}

\begin{proof}
Let $D $ be the canonical connection of $(\hat J_1,\hat J_2,\hat J_1\hat J_2)$.
Then, for any $X,Y\in \Gamma(TM)$ and $\eta,\beta\in \Gamma(T^*M)$, we have
\[
(D _{X+\eta}\hat J_{a,b})(Y+\beta):=D _{X+\eta}(\hat J_{a,b}(Y+\beta))-\hat J_{a,b}(D _{X+\eta}Y+\beta)=
\]
\[
=a(D _{X+\eta}\hat J_1)(Y+\beta)+b\Big((D _{X+\eta}\hat J_1)(\hat J_2(Y+\beta))+\hat J_1((D _{X+\eta}\hat J_2)(Y+\beta))\Big),
\]
and the conclusion follows from Theorem \ref{HJP}.
\end{proof}

\begin{proposition}
Let $(\hat K_1, \hat K_2,\hat K_1\hat K_2)$ be a generalized para-quaternionic structure on $M$, let $(\hat K_{a,b},\hat K_1)_{a,b\in \mathbb R, \ a^2-b^2=\pm 1}$ and let $(\hat K_{a,b}, \hat K_2)_{a,b\in \mathbb R, \ a^2-b^2=\pm 1}$ as above.
Then the canonical connections of $(\hat K_1,\hat K_2)$, $(\hat K_{a,b},\hat K_1)$ and $(\hat K_{a,b},\hat K_2)$ coincide.
\end{proposition}

\begin{proof}
Let $D $ be the canonical connection of $(\hat K_1,\hat K_2,\hat K_1\hat K_2)$.
Then, for any $X,Y\in \Gamma(TM)$ and $\eta,\beta\in \Gamma(T^*M)$, we have
\[
(D _{X+\eta}\hat K_{a,b})(Y+\beta):=D _{X+\eta}(\hat K_{a,b}(Y+\beta))-\hat K_{a,b}(D _{X+\eta}Y+\beta)=
\]
\[
=a(D _{X+\eta}\hat K_1)(Y+\beta)+b\Big((D _{X+\eta}\hat K_1)(\hat K_2(Y+\beta))+\hat K_1((D _{X+\eta}\hat K_2)(Y+\beta))\Big),
\]
and the conclusion follows from Theorem \ref{hjp}.
\end{proof}

\bigskip

\noindent Adara M. BLAGA, \\
Department of Mathematics, \\
Faculty of Mathematics and Computer Science, \\
West University of Timi\c{s}oara, \\
Bld. V. P\^{a}rvan, 4-6, 300223, Timi\c{s}oara, Rom\^{a}nia. \\
Email: adarablaga@yahoo.com

\bigskip

\noindent Antonella NANNICINI, \\
Department of Mathematics and Informatics, \\
"U. Dini", University of Florence,\\
Viale Morgagni, 67/a, 50134, Firenze, Italy.\\
Email: antonella.nannicini@unifi.it

\end{document}